\documentclass[pdflatex, sn-mathphys-num]{sn-jnl}

\usepackage{graphicx}%
\usepackage{multirow}%
\usepackage{amsmath, amssymb, amsfonts}%
\usepackage{amsthm}%
\usepackage{mathrsfs}%
\usepackage[title]{appendix}%
\usepackage{xcolor}%
\usepackage{textcomp}%
\usepackage{manyfoot}%
\usepackage{booktabs}%
\usepackage{algorithm}%
\usepackage{algorithmicx}%
\usepackage{algpseudocode}%
\usepackage{listings}%

\theoremstyle{thmstyleone}%
\newtheorem{theorem}{Theorem}
\theoremstyle{thmstyletwo}%
\newtheorem{example}{Example}%

\theoremstyle{thmstylethree}%
\newtheorem{lemma}{Lemma}
\newtheorem{corollary}{Corollary}
\newtheorem{assumption}{Assumption}

\DeclareMathOperator{\Ls}{L}
\DeclareMathOperator{\LLs}{\mathbf L}
\DeclareMathOperator{\Hs}{H}

\DeclareMathOperator{\Cs}{C}

\DeclareMathOperator{\Ds}{D}

\DeclareMathOperator{\Us}{U}
\DeclareMathOperator{\Vs}{V}

\newcommand{\norm}[1]{\left\lVert{#1}\right\rVert}

\newcommand{\inprod}[1]{\left\langle{#1}\right\rangle}

\newcommand{\vertiii}[1]{{\left\vert\kern-0.25ex\left\vert\kern-0.25ex\left\vert #1 
    \right\vert\kern-0.25ex\right\vert\kern-0.25ex\right\vert}}
    
\newcommand{\bm}[1]{\boldsymbol{#1}}

\newcommand{\fa}{\forall}

\newcommand{\xb}{\bm{x}}

\newcommand{\dx}{\,\mathrm{d}\xb}

\newcommand{\dt}{\,\mathrm{d}t}

\newcommand{\q}{\quad}

\newcommand{\qqq}{\qquad\quad}
\newcommand{\qqqq}{\qquad\qquad}

\begin{document}

\title[Error analysis of a space-time finite element method]
{Improved a priori error estimates for a space-time finite element method for parabolic problems}

\author*[1]{\fnm{Thi Thanh Mai} \sur{Ta}}\email{mai.tathithanh@hust.edu.vn}

\author[1]{\fnm{Quang Huy} \sur{Nguyen}}\email{huy.nguyenquang1@hust.edu.vn}

\author[1]{\fnm{Phi Hung} \sur{Pham}}\email{hung.pp216836@sis.hust.edu.vn}

\affil[1]{\orgdiv{Faculty of Mathematics and Informatics}, \orgname{Hanoi University of Science and Technology}, \orgaddress{\street{Dai Co Viet}, \city{Hanoi}, \postcode{11657}, \country{Viet Nam}}}


\abstract{In this paper, we employ a space-time finite element method to discretize the parabolic initial-boundary value problem and extend its error analysis with refined estimates on unstructured space-time meshes. We establish higher-order estimates in three different norms, thereby supplementing existing research. Moreover, we obtain an optimal estimate in a norm stronger than that of the trial space. Finally, we present numerical examples to illustrate our theoretical results.}





\maketitle

\section{Introduction}

Let $\Omega$ be a bounded, Lipschitz domain in $\mathbb{R}^n\ \left(n=1,2,3\right)$ with boundary $\partial\Omega$ and let $\left(0,T\right)$ be a bounded time interval, where $T>0$. The space-time cylinder is denoted by $\Omega_T:= \Omega\times \left(0, T\right)$. This paper considers the following parabolic initial-boundary value problem
\begin{equation}
    \label{eq: PDE problem}
    \begin{cases}
        \partial_t u - \nabla \cdot\left(\bm{A} \nabla u\right) = f \qqq & \text{in} \q \Omega_T,\\
        u = 0 \qqq  & \text{on} \q \partial\Omega \times \left(0, T\right),\\
        u (\cdot, 0) = 0 \qqq & \text{on} \q \Omega,
    \end{cases}
\end{equation}
where $\bm{A} = \bm{A}^{\top}\in \mathbb{L}^\infty\left(\Omega_T\right)$ is a given uniformly positive definite matrix, and $f\in \Ls^2\left(\left(0,T\right), \Hs^{-1}\left(\Omega\right)\right)$ represents the source term. Problems of this form arise in various physical and industrial applications, including heat conduction, population dispersion, and polymer thermal diffusivity.

In the past 50 years, space-time finite element methods have emerged as one of the most efficient approaches for solving Problem \eqref{eq: PDE problem} on fully unstructured space-time meshes \cite{HH1988, Neumuller2013, BVZ2017}. Unlike the time-stepping schemes or the time-discontinuous Galerkin methods, this approach treats the time variable as an additional spatial variable and the temporal derivative of the solution as a convection term along the time direction. In this way, space and time can be discretized simultaneously, allowing us to enhance the rapid development of parallel computations for tackling complex problems, such as optimal control \cite{LST+2021}, electromagnetics \cite{GGS2024}, and moving-subdomains problems \cite{NLPT2024}. 

Regarding the error analysis of space-time finite element methods, let us mention some recent works \cite{Steinbach2015, LMN2016, Moore2018, SW2021}. In \cite{Steinbach2015}, Steinbach proposed a conforming space-time Galerkin approximation for Problem \eqref{eq: PDE problem} and established an error estimate in a discrete trial space norm. Subsequently, the theoretical analysis of an upwind-stabilized space-time method for the heat equation was studied in \cite{LMN2016, Moore2018}. The authors ended up with a priori error estimates with respect to discrete test space norms. Moreover, utilizing the least square formulation of parabolic initial-boundary value problems, Stevenson and Westerdiep arrived at a stable approximation and a quasi-optimal error estimate in the trial space norm \cite{SW2021}. However, these studies primarily focused on errors measured in specific norms on the trial or test space, lacking results with other miscellaneous norms of higher order as in the literature on classical time discretizations \cite{Thomee2006}.

Supplementing the work of Steinbach \cite{Steinbach2015}, this paper analyzes the error of the space-time finite element method in four different norms. To improve the convergence order of this method and derive the desired error estimates, we extend the application of duality arguments from conforming approximations of elliptic problems \cite{Nitsche1968} to that of Problem \eqref{eq: PDE problem}. We obtain higher-order estimates in the $\Ls^2\left(\Omega\right)$-norm at $t=T$ and the $\Ls^2\left(\Omega_T\right)$-norm, which is weaker than the norm considered in \cite{Steinbach2015}. Moreover, with respect to a norm stronger than that of the trial space, specifically the $\Hs^1\left(\Omega_T\right)$-norm, we establish an optimal error estimate. The convergence order can be further improved when employing a negative-order norm. To the best of our knowledge, these error estimates have not yet been explored in the numerical analysis literature on space-time finite element methods.

The manuscript is structured into four sections. In the \hyperref[sec: variational formulation and its discretization]{next} section, we present the Petrov-Galerkin variational formulation of Problem \eqref{eq: PDE problem} and its space-time discretization, along with the approximability of the $\Ls^2$-orthogonal projection. In Section \ref{sec: a priori error estimates}, we establish and prove the improved error estimates. Subsequently, Section \ref{sec: numerical results} presents numerical experiments that validate the theoretical findings. Finally, we discuss perspectives and potential directions for future research.

\section{Variational formulation and its discretization}
\label{sec: variational formulation and its discretization}

Throughout this work, we use standard notations for Lebesgue spaces, Sobolev spaces, and Bochner spaces, along with their associated norms. For further details, see \cite[Appendix B.1]{EG2004} and \cite[Section 64.1]{Ern2021c}, for example. The symbol $C>0$ denotes a generic constant that is independent of the solution $u$ and the mesh size $h$ but may depend on the space-time cylinder $\Omega_T$ and the matrix $\bm{A}$. Its value may vary across different contexts.

Define $\Vs := \Ls^2\left(\left(0,T\right), \Hs_0^1\left(\Omega\right)\right)$ for convenience and endow this space with an equivalent norm
$$
\norm{y}_{\Vs}^2 := \int\limits_0^T \int\limits_\Omega \left(\bm{A}\nabla y\right)\cdot\nabla y \dx \dt \qqqq \forall y \in \Vs.
$$
This equivalence follows from the uniformly positive definiteness of $\bm{A}$ and the Poincar{\' e}–Steklov inequality 
\begin{equation}
    \label{eq: compare L2 and the Y norm}
    \norm{y}_{\Ls^2\left(\Omega_T\right)}\le C\norm{\nabla y}_{\LLs^2\left(\Omega_T\right)}\qqqq \forall y\in \Vs,
\end{equation}
see \cite[Lemma B.61]{EG2004}. The notation $\Vs^\prime$ stands for the dual space of $\Vs$. We denote by $\inprod{\cdot, \cdot}_{\Vs^\prime  \times \Vs}$ the duality pairing between $\Vs^\prime$ and $\Vs$. Let us introduce the spaces 
$$
\Us := \left\{y \in \Vs \mid \partial_t y \in \Vs^\prime \right\}\qqq \text{and}\qqq \Us_0 := \left\{y \in \Us \mid y(\cdot, 0) = 0 \right\},
$$
equipped with the graph norm 
$$
\norm{y}^2_{\Us} := \norm{y}_{\Vs}^2 + \norm{\partial_t y}_{\Vs^\prime }^2 \qqqq \forall y \in \Us.
$$
In $\Us$, the trace operator $y\in \Us\rightarrow y\left(\cdot,t\right)\in \Ls^2\left(\Omega\right)$ is bounded for almost every $t\in \left[0,T\right]$. Moreover, we recall from \cite[Lemma 64.40]{Ern2021c} the following inequality
\begin{equation}
    \label{eq: time trace inequality}
    \sup_{t\in \left[0,T\right]} \norm{y\left(\cdot,t\right)}_{\Ls^2\left(\Omega\right)}\le C\norm{y}_{\Us}\qqqq \forall y\in \Us.
\end{equation} 

Given a source term $f\in \Vs^\prime$, the Petrov-Galerkin variational formulation of Problem \eqref{eq: PDE problem} is stated as follows: Find $u\in \Us_0$ such that
\begin{equation}
    \label{eq: variational formulation}
    a\left(u, v\right) = \inprod{f, v}_{\Vs^\prime  \times \Vs} \qqqq \forall v \in \Vs,
\end{equation}
where the bilinear form $a: \Us_0\times \Vs \to \mathbb{R}$ is defined as
$$
a\left(u, v\right) := \inprod{\partial_t u, v}_{\Vs^\prime  \times \Vs} + \int\limits_0^T \int\limits_\Omega \left(\bm{A}\nabla u\right) \cdot \nabla v \dx\dt \qqqq \forall \left(u,v\right)\in \Us_0\times \Vs.
$$
By the Banach-Ne{\v c}as-Babu{\v s}ka theorem \cite[Theorem 2.6]{EG2004}, the well-posedness of this problem follows from the boundedness of $a\left(\cdot,\cdot\right)$, the inf-sup condition \cite[Theorem 2.1]{Steinbach2015}, as well as a result in \cite[Section 2]{LST+2021}. 

Let $\Omega$ be a polyhedron in $\mathbb{R}^n\ \left(n=1,2,3\right)$. Consider a family of quasi-uniform meshes $\left\{\mathcal{T}_h\right\}_{h\in \left(0,h^\ast\right)}$ of the cylinder $\Omega_T = \Omega\times \left(0,T\right)$, with mesh size $h\in \left(0,h^\ast\right)$, where $h^\ast>0$ is a fixed constant. We define the finite element space as
$$
\Us_h :=\left\{\varphi_h \in \Cs\left(\overline{\Omega_T}\right)\mid \varphi_{h\, \mid \, K}\in \mathbb{P}_k\left(K\right) \text{ for all } K\in \mathcal{T}_h\right\} \cap \Us_0.
$$
Here, for $k\in \mathbb{N}^\ast$ and $K\in \mathcal{T}_h$, let $\mathbb{P}_k\left(K\right)$ denote the space of polynomials of degree $k$ on $K$. The discrete counterpart of Problem \eqref{eq: variational formulation} is given by: Find $u_h\in \Us_h$ such that
\begin{equation}
    \label{eq: discrete variational formulation}
    a\left(u_h, v_h\right) = \inprod{f, v_h}_{\Vs^\prime  \times \Vs} \qqqq \forall v_h \in \Us_h.
\end{equation}
From \cite[Theorem 3.1]{Steinbach2015}, it follows that the bilinear form $a\left(\cdot,\cdot\right)$ satisfies the discrete stability condition
\begin{equation}
    \label{eq: discrete stability estimate}
    \sup_{v_h \in \Us_{h}\setminus \{0\}} \dfrac{a\left(u_h, v_h\right)}{\norm{v_h}_{\Vs}} \ge \dfrac{1}{2\sqrt{2}}\norm{u_h}_{h} \qqqq \fa u_h \in \Us_{h},
\end{equation}
where $\norm{\cdot}_h$ is a mesh-dependent norm on $\Us_0$, defined as
$$
\norm{y}_h^2 := \norm{y}_{\Vs}^2 + \norm{q_h\left(y\right)}_{\Vs}^2 \qqqq\forall y\in \Us_0,
$$
and $q_h\left(y\right)\in \Us_h$ is the unique solution to the problem
$$
\int\limits_0^T \int\limits_\Omega \left[\bm{A}\nabla q_h\left(y\right)\right]\cdot \nabla \phi_h \dx\dt = \inprod{\partial_t y, \phi_h}_{\Vs^\prime  \times \Vs} \qqqq \forall \phi_h \in \Us_h.
$$
Using \eqref{eq: discrete stability estimate} and the discrete Banach-Ne{\v c}as-Babu{\v s}ka theorem \cite[Theorem 2.22]{EG2004}, we conclude the well-posedness of Problem \eqref{eq: discrete variational formulation}. Moreover, the following Galerkin orthogonality holds
\begin{equation}
    \label{eq: Galerkin orthogonal}
    a\left(u-u_h,v_h\right)=0 \qqqq \forall v_h \in \Us_h.
\end{equation}

We end this section by discussing the approximability of the projection operator $\Pi_h: \Ls^1\left(\Omega_T\right) \to \Us_h$, which is defined by
$$
\int\limits_0^T \int\limits_\Omega \left(\Pi_h y - y\right) \varphi_h \dx\dt = 0 \qqqq \forall \varphi_h \in \Us_h,
$$
for all $y\in \Ls^1\left(\Omega_T\right)$. The following lemma follows from \cite[Remark 12.17, Section 22.5]{Ern2021}.

\begin{lemma}
    \label{lem: projection error}
    For all $y\in \Hs^1\left(\Omega_T\right)$, the operator $\Pi_h$ satisfies the inequality
    $$
    \norm{\left(y-y_h\right)\left(\cdot,T\right)}_{\Ls^2\left(\Omega\right)} \le C \sqrt{h}\norm{\Ds y}_{\LLs^2\left(\Omega_T\right)}.
    $$
    Moreover, for all $\eta\in \left[1, k+1\right]$ and $y\in \Hs^\eta\left(\Omega_T\right)$, the following estimate holds
    $$
    \norm{y - \Pi_h y}_{\Ls^2\left(\Omega_T\right)} + h\norm{\Ds\left(y- \Pi_h y\right)}_{\LLs^2\left(\Omega_T\right)} \le C h^\eta\norm{y}_{\Hs^\eta\left(\Omega_T\right)}.
    $$
    Here, $\Ds:= \left(\nabla, \partial_t\right)^\top$ denotes the space-time gradient operator.
\end{lemma}

\section{A priori error estimates}
\label{sec: a priori error estimates}

Let $u\in \Us_0$ and $u_h\in \Us_h$ be the solutions to Problems \eqref{eq: variational formulation} and \eqref{eq: discrete variational formulation}, respectively. In this section, we estimate the error $u-u_h$ in various norms. We begin by recalling an error estimate in the norm $\norm{\cdot}_h$. For the proof when $k=1$, see \cite[Theorem 3.3, Corollary 3.4]{Steinbach2015}. The result for $k>1$ follows similarly.

\begin{lemma}
    \label{lem: error estimate h-norm}
    Let $u\in \Us_0$ and $u_h\in \Us_h$ be the solutions to Problems \eqref{eq: variational formulation} and \eqref{eq: discrete variational formulation}, respectively. For $s\in \left[1,k+1\right]$, if $u\in \Hs^s\left(\Omega_T\right)$, then we have
    $$
    \norm{u-u_h}_h \le C h^{s-1}\norm{u}_{\Hs^s\left(\Omega_T\right)}.
    $$
\end{lemma}

The primary objective of this work is to establish higher-order error estimates beyond the bound in Lemma \ref{lem: error estimate h-norm}. To achieve this, we employ duality arguments \cite{Nitsche1968}. For any $g\in \Ls^2\left(\Omega_T\right)$ and $z_T\in \Hs^1_0\left(\Omega\right)$, it follows from \cite[Corollary 2.3]{Steinbach2015} and \cite[Proposition 1]{MS2020} that the following problem
\begin{equation}
    \label{eq: Aubin-Nitshe assumption}
    \begin{cases}
        -\inprod{\partial_t z, w}_{\Vs^\prime  \times \Vs} + \displaystyle\int\limits_0^T \int\limits_\Omega \left(\bm{A}\nabla z\right)\cdot\nabla w\dx\dt = \displaystyle\int\limits_0^T \int\limits_\Omega g w\dx\dt \qqq \forall w\in \Vs, \\
        z\left(\cdot,T\right) = z_T 
    \end{cases}
\end{equation}
admits a solution $z\in \Us$, which further satisfies $z\in \Hs^1\left(\Omega_T\right)$ along with the stability estimate
\begin{equation}
    \label{eq: regularity estimate}
    \norm{z}_{\Hs^1\left(\Omega_T\right)}\le C\left(\norm{g}_{\Ls^2\left(\Omega_T\right)} + \norm{z_T}_{\Ls^2\left(\Omega\right)}\right).
\end{equation}

Next, we choose $w= u-u_h \in \Us_0$ in \eqref{eq: Aubin-Nitshe assumption}, integrate by parts with $z\in \Us$ and $u-u_h\in \Us_0$, and apply \eqref{eq: Galerkin orthogonal} to get
$$
\int\limits_0^T\int\limits_\Omega g \left(u-u_h\right)\dx\dt + \int\limits_\Omega z_T\left(u-u_h\right)\left(\xb, T\right)\dx = a\left(u-u_h, z\right) = a\left(u-u_h, e\right),
$$
where we define $e:= z - \Pi_h z$. Integrating by parts once more with $u-u_h\in \Us_0$ and $ e\in \Hs^{1}\left(\Omega_T\right)$, and subsequently applying \eqref{eq: compare L2 and the Y norm}, we obtain
$$
\begin{aligned}
    &a\left(u-u_h, e\right)= \\
    &= \int\limits_\Omega \left(u-u_h\right)\left(\xb, T\right) e\left(\xb, T\right)\dx + \int\limits_0^T\int\limits_\Omega -\left(u-u_h\right)\left(\partial_t e\right) +  \left[\bm{A} \nabla \left(u-u_h\right) \right]\cdot \nabla e\dx \dt\\
    &\le \norm{\left(u-u_h\right)\left(\cdot, T\right)}_{\Ls^2\left(\Omega\right)}\norm{ e\left(\cdot, T\right)}_{\Ls^2\left(\Omega\right)} + C\norm{\nabla\left(u-u_h\right)}_{\LLs^2\left(\Omega_T\right)}\norm{\Ds e}_{\LLs^2\left(\Omega_T\right)}.
\end{aligned}
$$
Thus, we end up with the following inequality
\begin{equation}
    \label{eq: key estimate}
    \begin{aligned}
        &\int\limits_0^T\int\limits_\Omega g \left(u-u_h\right)\dx\dt + \int\limits_\Omega z_T\left(u-u_h\right)\left(\xb, T\right)\dx \le\\
        &\le \norm{\left(u-u_h\right)\left(\cdot, T\right)}_{\Ls^2\left(\Omega\right)}\norm{ e\left(\cdot, T\right)}_{\Ls^2\left(\Omega\right)} + C\norm{\nabla\left(u-u_h\right)}_{\LLs^2\left(\Omega_T\right)}\norm{\Ds e}_{\LLs^2\left(\Omega_T\right)},
    \end{aligned}
\end{equation}
for all $g\in \Ls^2\left(\Omega_T\right)$ and $z_T\in \Hs^1_0\left(\Omega\right)$. This inequality plays a fundamental role in the subsequent error analysis. To bound the right-hand side of \eqref{eq: key estimate}, we introduce the following assumption:

\begin{assumption}
    \label{assump: Aubin-Nitshe assumption}
    Assume that the solution $z\in \Us$ to Problem \eqref{eq: Aubin-Nitshe assumption} possesses additional regularity, namely $z\in \Hs^2\left(\Omega_T\right)$.
\end{assumption}

This assumption was previously employed in \cite{LMN2016, Moore2018} and validated in \cite[Remark 3.3]{LST+2021b}. We are now ready to establish our main theoretical results. First, we derive the $\Ls^2\left(\Omega\right)$-norm error estimate at $t=T$. In the following theorem, let $z_1$ be the solution to Problem \eqref{eq: Aubin-Nitshe assumption} when $g=0$ and $z_T=\norm{\left(u-u_h\right)\left(\cdot, T\right)}^{-1}_{\Ls^2\left(\Omega\right)}\left(u-u_h\right)\left(\cdot,T\right)$.

\begin{theorem}
    \label{theo: error estimate L2 T norm}
    Let $u\in \Us_0$ and $u_h\in \Us_h$ be the solutions to Problems \eqref{eq: variational formulation} and \eqref{eq: discrete variational formulation}, respectively. Suppose that $u\in \Hs^s\left(\Omega_T\right)$ for some $s\in \left[1,k+1\right]$ and that Assumption \ref{assump: Aubin-Nitshe assumption} is fullfilled. Then, we can find $h^\ast>0$ such that for all $h\in \left(0,h^\ast\right)$, the following estimate holds
    $$
    \norm{\left(u-u_h\right)\left(\cdot, T\right)}_{\Ls^2\left(\Omega\right)} \le C h^{s} \norm{u}_{\Hs^s\left(\Omega_T\right)}\norm{z_1}_{\Hs^2\left(\Omega_T\right)}.
    $$
\end{theorem}

\begin{proof}
    By substituting $g=0$ and $z_T=\norm{\left(u-u_h\right)\left(\cdot, T\right)}^{-1}_{\Ls^2\left(\Omega\right)}\left(u-u_h\right)\left(\cdot,T\right)$ into \eqref{eq: key estimate}, and then invoking Lemmas \ref{lem: projection error} and \ref{lem: error estimate h-norm}, we obtain
    $$
    \begin{aligned}
        & \norm{\left(u-u_h\right)\left(\cdot, T\right)}_{\Ls^2\left(\Omega\right)} \le \\
        &\le \norm{\left(u-u_h\right)\left(\cdot, T\right)}_{\Ls^2\left(\Omega\right)}\norm{e_1\left(\cdot, T\right)}_{\Ls^2\left(\Omega\right)} + C h^{s-1} \norm{u}_{\Hs^s\left(\Omega_T\right)}h\norm{z_1}_{\Hs^2\left(\Omega_T\right)},
    \end{aligned}
    $$
    where $e_1 = z_1 - \Pi_h z_1$. Next, we apply Lemma \ref{lem: projection error} once more and use \eqref{eq: regularity estimate} to get
    $$
    \norm{e_1\left(\cdot, T\right)}_{\Ls^2\left(\Omega\right)} \le C \sqrt{h}\norm{\Ds z_1}_{\LLs^2\left(\Omega_T\right)} \le C \sqrt{h}.
    $$
    Thus, we arrive at the inequality
    $$
    \left(1- C \sqrt{h}\right)\norm{\left(u-u_h\right)\left(\cdot, T\right)}_{\Ls^2\left(\Omega\right)} \le C h^s \norm{u}_{\Hs^s\left(\Omega_T\right)}\norm{z_1}_{\Hs^2\left(\Omega_T\right)}.
    $$
    Since $h\in \left(0,h^\ast \right)$ for some $h^\ast>0$, we complete the proof by choosing $h^\ast$ such that $1- C \sqrt{h} > 1- C \sqrt{h^\ast} \ge \frac{1}{2}$.
\end{proof}

We now derive a higher-order error estimate in the $\Ls^2\left(\Omega_T\right)$-norm. From \eqref{eq: compare L2 and the Y norm}, we observe that this norm is weaker than $\norm{\cdot}_{\Vs}$, and consequently, weaker than $\norm{\cdot}_{h}$. Such an estimate is essential for the error analysis of optimal control and eigenvalue problems. In the following theorem, we consider Problem \eqref{eq: Aubin-Nitshe assumption} with $g=\norm{u-u_h}^{-1}_{\Ls^2\left(\Omega_T\right)}\left(u-u_h\right)$ and $z_T=0$, and denote by $z_2$ its corresponding solution.
    
\begin{theorem}
    \label{theo: error estimate L2 Q norm}
    Let $u\in \Us_0$ and $u_h\in \Us_h$ be the solutions to Problems \eqref{eq: variational formulation} and \eqref{eq: discrete variational formulation}, respectively. Suppose that $u\in \Hs^s\left(\Omega_T\right)$ for some $s\in \left[1,k+1\right]$ and that Assumption \ref{assump: Aubin-Nitshe assumption} is satisfied. Then, there exists $h^\ast>0$ such that for all $h\in \left(0,h^\ast\right)$, we have the following estimate
    $$
    \norm{u-u_h}_{\Ls^2\left(\Omega_T\right)} \le C h^s \norm{u}_{\Hs^s\left(\Omega_T\right)}\left(\norm{z_1}_{\Hs^2\left(\Omega_T\right)} + \norm{z_2}_{\Hs^2\left(\Omega_T\right)}\right).
    $$
\end{theorem}

\begin{proof}
    When $g=\norm{u-u_h}^{-1}_{\Ls^2\left(\Omega_T\right)}\left(u-u_h\right)$ and $z_T=0$, it follows from \eqref{eq: key estimate} that 
    $$
    \begin{aligned}
        &\norm{u-u_h}_{\Ls^2\left(\Omega_T\right)} \le\\
        &\le \norm{\left(u-u_h\right)\left(\cdot, T\right)}_{\Ls^2\left(\Omega\right)}\norm{\left(\Pi_h z_2\right)\left(\cdot, T\right)}_{\Ls^2\left(\Omega\right)} + C\norm{\nabla\left(u-u_h\right)}_{\LLs^2\left(\Omega_T\right)}\norm{\Ds e_2}_{\LLs^2\left(\Omega_T\right)},
    \end{aligned}
    $$
    where $e_2 = z_2 - \Pi_h z_2$. Now, by invoking \eqref{eq: time trace inequality}, \eqref{eq: compare L2 and the Y norm}, the $\Hs^1\left(\Omega_T\right)$-seminorm stability of $\Pi_h$ \cite[Proposition 22.21]{Ern2021}, and finally \eqref{eq: regularity estimate}, we obtain
    $$
    \begin{aligned}
        \norm{\left(\Pi_h z_2\right)\left(\cdot, T\right)}^2_{\Ls^2\left(\Omega\right)} &\le C\norm{\Pi_h z_2}^2_{\Us} \\
        &\le C\left(\norm{\nabla\left(\Pi_h z_2\right)}^2_{\LLs^2\left(\Omega_T\right)}+\norm{\Ds\left(\Pi_h z_2\right)}^2_{\LLs^2\left(\Omega_T\right)}\right)\\
        &\le C\norm{\Ds\left(\Pi_h z_2\right)}^2_{\LLs^2\left(\Omega_T\right)} \\
        &\le C \norm{\Ds z_2}^2_{\LLs^2\left(\Omega_T\right)} \le C.
    \end{aligned}
    $$
    Together with Theorem \ref{theo: error estimate L2 T norm}, as well as Lemmas \ref{lem: projection error} and \ref{lem: error estimate h-norm}, we thus derive the desired result.
\end{proof}

We proceed with estimating the error in the $\Hs^1\left(\Omega_T\right)$-norm. This norm is stronger than $\norm{\cdot}_{h}$ due to the inclusion of $\norm{\partial_t \cdot}_{\Ls^2\left(\Omega_T\right)}$. Unlike Lemma \ref{lem: error estimate h-norm}, this estimate cannot be derived directly from \eqref{eq: discrete stability estimate}. Following classical arguments for elliptic problems, we combine Lemmas \ref{lem: projection error} and \ref{theo: error estimate L2 Q norm} with Theorem \ref{theo: error estimate L2 Q norm} to establish the desired result.

\begin{theorem}
    \label{theo: error estimate H1 Q norm}
    Let $u\in \Us_0$ and $u_h\in \Us_h$ be the solutions to Problems \eqref{eq: variational formulation} and \eqref{eq: discrete variational formulation}, respectively. Assume that $u\in \Hs^s\left(\Omega_T\right)$ for some $s\in \left[1,k+1\right]$ and that Assumption \ref{assump: Aubin-Nitshe assumption} holds. Then, we can find $h^\ast>0$ such that for all $h\in \left(0,h^\ast\right)$, their holds the following estimate
    \begin{equation}
        \label{eq: error estimate H1 Q norm}
        \norm{u-u_h}_{\Hs^1\left(\Omega_T\right)} \le C h^{s-1}\norm{u}_{\Hs^s\left(\Omega_T\right)}\left(1+\norm{z_1}_{\Hs^2\left(\Omega_T\right)} + \norm{z_2}_{\Hs^2\left(\Omega_T\right)}\right).
    \end{equation}
\end{theorem}

\begin{proof}
    Using the triangle inequality, we have
    $$
    \norm{u-u_h}_{\Hs^1\left(\Omega_T\right)} \le \norm{u-\Pi_h u}_{\Hs^1\left(\Omega_T\right)} + \norm{\Pi_h u-u_h}_{\Hs^1\left(\Omega_T\right)}.
    $$
    On the one hand, it follows from Lemma \ref{lem: projection error} that 
    $$
    \norm{u-\Pi_h u}_{\Hs^1\left(\Omega_T\right)}\le Ch^{s-1}\norm{u}_{\Hs^s\left(\Omega_T\right)}.
    $$
    On the other hand, we apply the global inverse inequality \cite[Corollary 1.141]{EG2004}, Lemma \ref{lem: projection error}, and Theorem \ref{theo: error estimate L2 Q norm} to obtain
    $$
    \begin{aligned}
        \norm{\Pi_h u-u_h}_{\Hs^1\left(\Omega_T\right)} &\le C h^{-1}\norm{\Pi_h u-u_h}_{\Ls^2\left(\Omega_T\right)} \\
        &\le C h^{-1}\left(\norm{\Pi_h u-u}_{\Ls^2\left(\Omega_T\right)} + \norm{u-u_h}_{\Ls^2\left(\Omega_T\right)}\right) \\
        &\le C h^{s-1}\norm{u}_{\Hs^s\left(\Omega_T\right)}\left(1+\norm{z_1}_{\Hs^2\left(\Omega_T\right)} + \norm{z_2}_{\Hs^2\left(\Omega_T\right)}\right),
    \end{aligned}
    $$
    for all $h\in \left(0,h^\ast\right)$, where $h^\ast > 0$ is given. The conclusion follows.
\end{proof}

From Lemma \ref{lem: error estimate h-norm} and Theorem \ref{theo: error estimate H1 Q norm}, we observe that the error estimates in the two norms $\norm{\cdot}_h$ and $\norm{\cdot}_{\Hs^1\left(\Omega_T\right)}$ are formally of the same order. The latter is preferable, as it evaluates the error with respect to $\norm{\partial_t \cdot}_{\Ls^2\left(\Omega_T\right)}$, thereby making more effective use of the condition $u\in \Hs^s\left(\Omega_T\right)$, where $s\in \left[1,k+1\right]$. 

However, it is worth noting that we can only conclude the convergence order from Theorem \ref{theo: error estimate H1 Q norm} if there exists an $h$-independent constant $C$ such that
\begin{equation}
    \label{eq: additional regularity estimate}
    \norm{z_1}_{\Hs^2\left(\Omega_T\right)} \le C \qqq \text{and}\qqq \norm{z_2}_{\Hs^2\left(\Omega_T\right)} \le C.
\end{equation}
In that case, we also obtain the convergence order with respect to the two $\Ls^2$-norms. Specifically, we have the following result:

\begin{corollary}
    \label{cor: order convergence in 3 norms}
    Let $u\in \Us_0$ and $u_h\in \Us_h$ be the solutions to Problems \eqref{eq: variational formulation} and \eqref{eq: discrete variational formulation}, respectively. Assume that the conditions of Theorem \ref{theo: error estimate H1 Q norm} and \eqref{eq: additional regularity estimate} hold. Then, as $h\to 0$, the following higher-order estimates hold
    $$
    \norm{\left(u-u_h\right)\left(\cdot, T\right)}_{\Ls^2\left(\Omega\right)} = \mathcal{O}\left(h^s\right) \qqq \text{and}\qqq \norm{u-u_h}_{\Ls^2\left(\Omega_T\right)} = \mathcal{O}\left(h^s\right),
    $$
    along with the optimal estimate
    $$
    \norm{u-u_h}_{\Hs^1\left(\Omega_T\right)} = \mathcal{O}\left(h^{s-1}\right).
    $$
\end{corollary} 

Finally, we improve the estimate in Theorem \ref{theo: error estimate L2 Q norm} by incorporating a negative-order norm. This refined estimate is particularly relevant to the study of superconvergence. On the space $\Ls^2\left(\Omega_T\right)$, we define the following norm
$$
\norm{y}_{\Hs^{-r }\left(\Omega_T\right)}:= \sup_{\xi\in \Hs^{r }\left(\Omega_T\right)}\dfrac{1}{\norm{\xi}_{\Hs^{r }\left(\Omega_T\right)}}\int\limits_0^T\int\limits_\Omega \xi y \dx\dt \qqqq \forall y \in \Ls^2\left(\Omega_T\right),
$$
for any fixed $r \ge 1$. Note that this norm is weaker than $\norm{\cdot}_{\Ls^2\left(\Omega_T\right)}$ and should not be confused with the norm of the dual space $\left(\Hs_0^r\left(\Omega_T\right)\right)^{\prime}$. To achieve the desired result, we employ duality arguments once again. Let us introduce another assumption:

\begin{assumption}
    \label{assump: H-r assumption}
    Let $k\ge 2$ and $r \in \left[1, k-1\right]$. For any $g\in \Hs^{r }\left(\Omega_T\right)$ and $z_T=0$, suppose that the solution $z\in \Us$ to Problem \eqref{eq: Aubin-Nitshe assumption} has an improved regularity $z\in \Hs^{r +2}\left(\Omega_T\right)$, along with the stability estimate $\norm{z}_{\Hs^{r +2}\left(\Omega_T\right)}\le C\norm{g}_{\Hs^{r }\left(\Omega_T\right)}$.
\end{assumption}

\begin{lemma}
    \label{lem: H-1 norm error estimate}
    Let $u\in \Us_0$ and $u_h\in \Us_h$ (with $k\ge 2$) be the solutions to Problems \eqref{eq: variational formulation} and \eqref{eq: discrete variational formulation}, respectively. For $r \in \left[1,k-1\right]$, if Assumption \ref{assump: H-r assumption} holds, then we have the following inequality
    $$
    \norm{u-u_h}_{\Hs^{-r }\left(\Omega_T\right)}\le Ch^{r +1}\left(\norm{\left(u-u_h\right)\left(\cdot, T\right)}_{\Ls^2\left(\Omega\right)} + \norm{u-u_h}_h\right).
    $$
\end{lemma}

\begin{proof}
    First, we choose $g\in \Hs^{r }\left(\Omega_T\right)\setminus\left\{0\right\}$ and $z_T=0$ in \eqref{eq: key estimate} to get
    $$
    \begin{aligned}
        &\int\limits_0^T\int\limits_\Omega g\left(u-u_h\right)\dx\dt \le \\
        &\le \norm{\left(u-u_h\right)\left(\cdot, T\right)}_{\Ls^2\left(\Omega\right)}\norm{e\left(\cdot, T\right)}_{\Ls^2\left(\Omega\right)} + C\norm{\nabla\left(u-u_h\right)}_{\LLs^2\left(\Omega_T\right)}\norm{\Ds e}_{\LLs^2\left(\Omega_T\right)}.
    \end{aligned}
    $$
    By applying \eqref{eq: time trace inequality} and \eqref{eq: compare L2 and the Y norm}, we have
    $$
    \norm{e\left(\cdot, T\right)}_{\Ls^2\left(\Omega\right)} \le C\norm{e}_{\Us} \le C \norm{e}_{\Hs^1\left(\Omega_T\right)}.
    $$
    Together with Lemma \ref{lem: projection error}, this implies
    $$
    \int\limits_0^T\int\limits_\Omega g\left(u-u_h\right)\dx\dt \le C\left(\norm{\left(u-u_h\right)\left(\cdot, T\right)}_{\Ls^2\left(\Omega\right)} + \norm{u-u_h}_h\right)h^{r +1}\norm{z}_{\Hs^{r + 2}\left(\Omega_T\right)},
    $$
    which leads to
    $$
    \begin{aligned}
        \norm{u-u_h}_{\Hs^{-r }\left(\Omega_T\right)} &= \sup_{g\in \Hs^{r }\left(\Omega_T\right)}\dfrac{1}{\norm{g}_{\Hs^{r }\left(\Omega_T\right)}}\int\limits_0^T\int\limits_\Omega g\left(u-u_h\right)\dx\dt \\
        &\le \dfrac{C\left(\norm{\left(u-u_h\right)\left(\cdot, T\right)}_{\Ls^2\left(\Omega\right)} + \norm{u-u_h}_h\right)h^{r +1}\norm{z}_{\Hs^{r  + 2}\left(\Omega_T\right)}}{\norm{g}_{\Hs^{r }\left(\Omega_T\right)}}\\
        &\le C\left(\norm{\left(u-u_h\right)\left(\cdot, T\right)}_{\Ls^2\left(\Omega\right)} + \norm{u-u_h}_h\right)h^{r +1},
    \end{aligned}
    $$
    using Assumption \ref{assump: H-r assumption}. This completes the proof.
\end{proof}

By combining Corollary \ref{cor: order convergence in 3 norms} with Lemmas \ref{lem: error estimate h-norm} and \ref{lem: H-1 norm error estimate}, we obtain the following $\Hs^{-r }\left(\Omega_T\right)$-norm convergence result:

\begin{theorem}
    Let $u\in \Us_0$ and $u_h\in \Us_h$ (with $k\ge 2$) be the solutions to Problems \eqref{eq: variational formulation} and \eqref{eq: discrete variational formulation}, respectively. Suppose that the assumptions of Corollary \ref{cor: order convergence in 3 norms} and Lemma \ref{lem: H-1 norm error estimate} hold. Then, as $h\to 0$, we achieve the following convergence order
    $$
    \norm{u-u_h}_{\Hs^{-r }\left(\Omega_T\right)} = \mathcal{O}\left(h^{r  + s}\right).
    $$
\end{theorem}

\section{Numerical results}
\label{sec: numerical results}

In this section, we present numerical examples to illustrate the improved error estimates discussed in Section \ref{sec: a priori error estimates}. Specifically, we compute the error in three norms: the $\Ls^2\left(\Omega\right)$-norm at $t=T$, the $\Ls^2\left(\Omega_T\right)$-norm, and the $\Hs^1\left(\Omega_T\right)$-norm. 

We employ continuous, elementwise linear finite elements to solve Problem \eqref{eq: discrete variational formulation} in the case where $\bm{A}$ is the identity matrix. The space-time domain $\Omega_T$ is discretized at various levels of mesh refinement, with $N$ subdivisions along each edge, starting from $N=2^2$. The corresponding mesh size is given by $h=1/N$. All experiments are conducted using FreeFEM++ \cite{Hecht2012}.

\begin{example}
    \label{ex: example 1}
    To verify our theoretical findings, we begin with \cite[Example 5.1]{Steinbach2015}, as this paper directly extends the error analysis presented in that work. We consider the space-time cylinder \(\Omega_T = (0,1)^2\) and choose the exact solution as
    $$
    u\left(x,t\right) = \sin\left(\pi x\right) \cos\left(\pi t\right),
    $$
    resulting in a non-homogeneous initial value problem. The initial value $u\left(\cdot, 0\right)$ is determined accordingly from this exact solution. 
\end{example}

Figure \ref{1d example plot in 2d and 3d} shows the solution $u_h$ of Problem \eqref{eq: discrete variational formulation} for $h=2^{-8}$. Table \ref{1d example 1} reports the errors and the experimental order of convergence for different mesh sizes. We observe that the error in all three norms decreases as the mesh is refined. At the same time, the corresponding convergence order tends to stabilize. The results presented in this table are further illustrated in Figure \ref{1d example 1 convergence}, which visually demonstrates the convergence behavior. 

These results align with the convergence rates established in Corollary \ref{cor: order convergence in 3 norms}. Furthermore, they suggest that those error estimates may also hold for more general problems with nonzero initial values.

\begin{figure}[h!]
    \centering
    \begin{minipage}{0.45\textwidth}
        \centering
        \includegraphics[width=\linewidth]{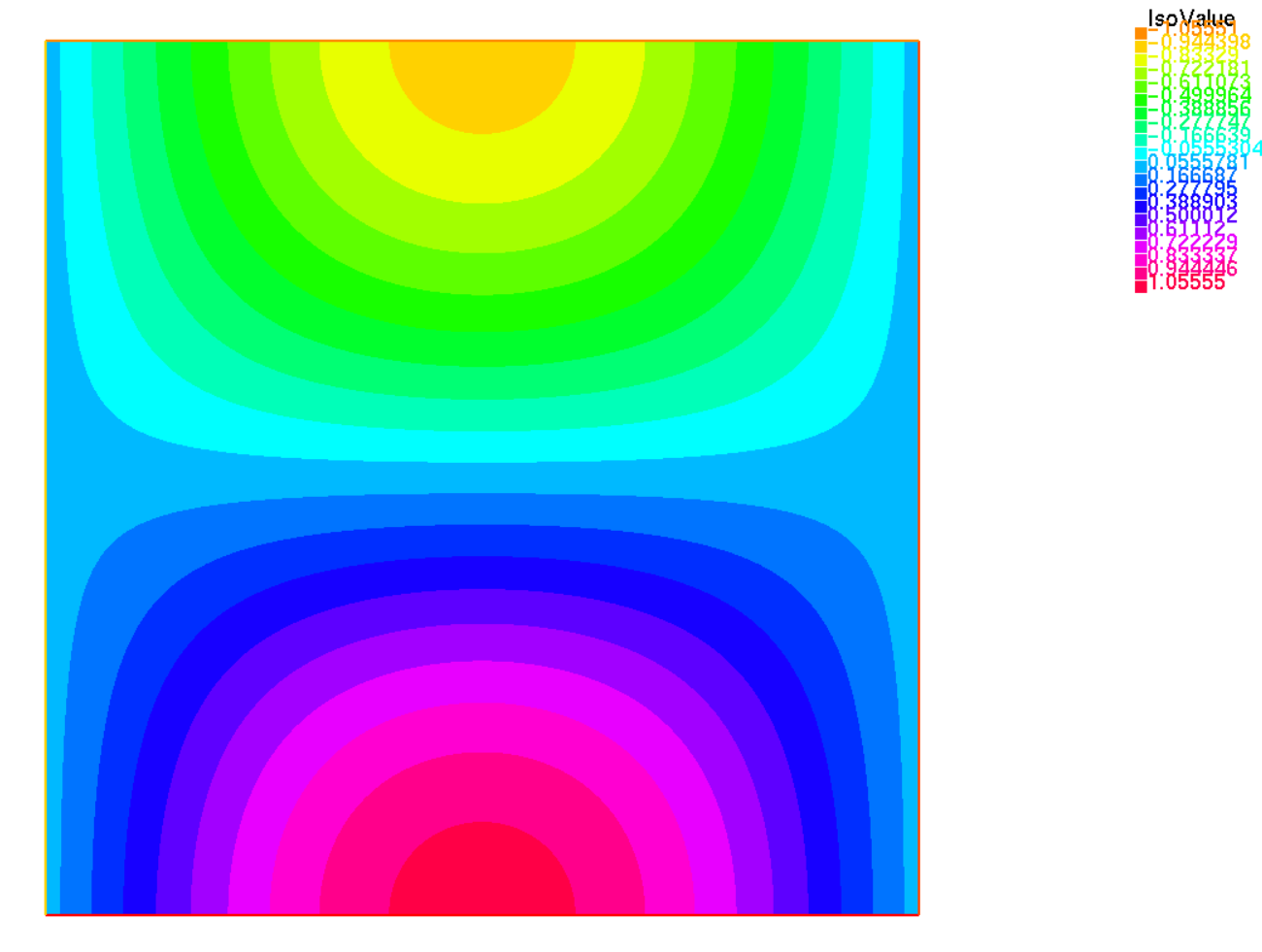}
    \end{minipage}%
    \hfill
    \begin{minipage}{0.45\textwidth}
        \centering
        \includegraphics[width=\linewidth]{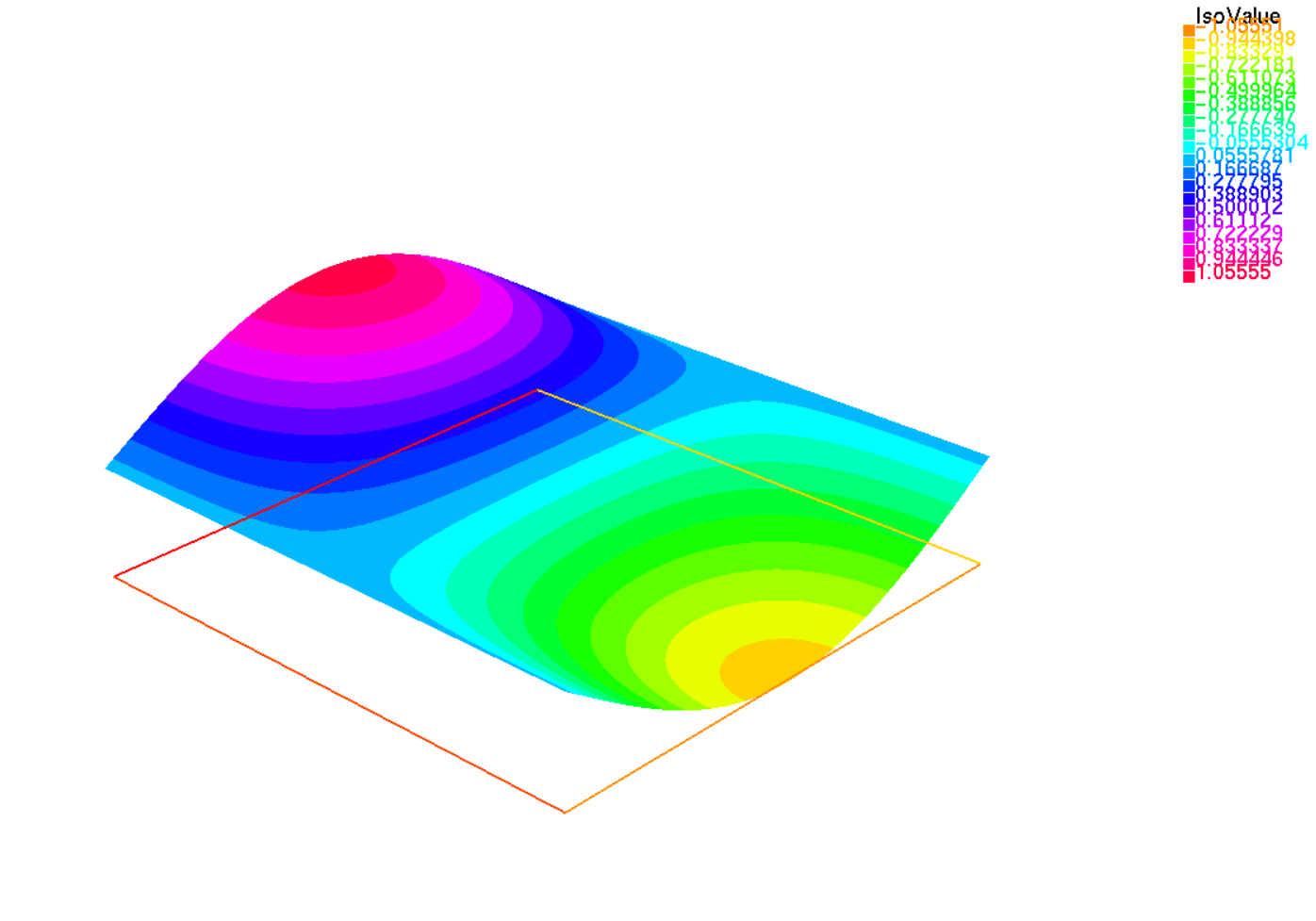}
    \end{minipage}
    \caption{The discrete solution $u_h$ for Example \ref{ex: example 1} with mesh size $h= 2^{-8}$.}
    \label{1d example plot in 2d and 3d}
\end{figure}

\begin{table}[!htp]
    \centering
    \resizebox{\textwidth}{!}{
    \begin{minipage}{\textwidth}%
        \caption{Errors in three different norms for various levels of mesh refinement in Example \ref{ex: example 1}.}
        \label{1d example 1}
        \begin{tabular}{|c|c|c|c|c|c|c|}
            \hline
            \rule{0pt}{2.75ex}  
             $h$ & \( \norm{\left(u-u_h\right)\left(\cdot, T\right)}_{\Ls^2\left(\Omega\right)} \) & Order & \( \norm{u-u_h}_{\Ls^2\left(\Omega_T\right)} \) & Order & \( \norm{u-u_h}_{\Hs^1\left(\Omega_T\right)} \) & Order \\[1ex] \hline
            \rule{0pt}{2.75ex}
            \(2^{-2}\) & \(9.969 \times 10^{-2}\) & -      & \(4.728 \times 10^{-2}\) & -      & \(3.055 \times 10^{-1}\) & -       \\ 
            \(2^{-3}\) & \(3.089 \times 10^{-2}\) & 1.690  & \(1.479 \times 10^{-2}\) & 1.677  & \(9.880 \times 10^{-2}\) & 1.629   \\ 
            \(2^{-4}\) & \(8.743 \times 10^{-3}\) & 1.821  & \(3.928 \times 10^{-3}\) & 1.913  & \(3.290 \times 10^{-2}\) & 1.586   \\ 
            \(2^{-5}\) & \(2.395 \times 10^{-3}\) & 1.868  & \(9.962 \times 10^{-4}\) & 1.979  & \(1.378 \times 10^{-2}\) & 1.256   \\ 
            \(2^{-6}\) & \(6.352 \times 10^{-4}\) & 1.915  & \(2.497 \times 10^{-4}\) & 1.996  & \(6.693 \times 10^{-3}\) & 1.041   \\ 
            \(2^{-7}\) & \(1.636 \times 10^{-4}\) & 1.957  & \(6.225 \times 10^{-5}\) & 2.004  & \(3.353 \times 10^{-3}\) & 0.997   \\ 
            \(2^{-8}\) & \(4.220 \times 10^{-5}\) & 1.955  & \(1.582 \times 10^{-5}\) & 1.977  & \(1.682 \times 10^{-3}\) & 0.996   \\ \hline
        \end{tabular}
    \end{minipage}}
\end{table}

\begin{figure}[h!]
    \centering
    \begin{minipage}{0.48\textwidth}
        \centering
        \includegraphics[width=\linewidth]{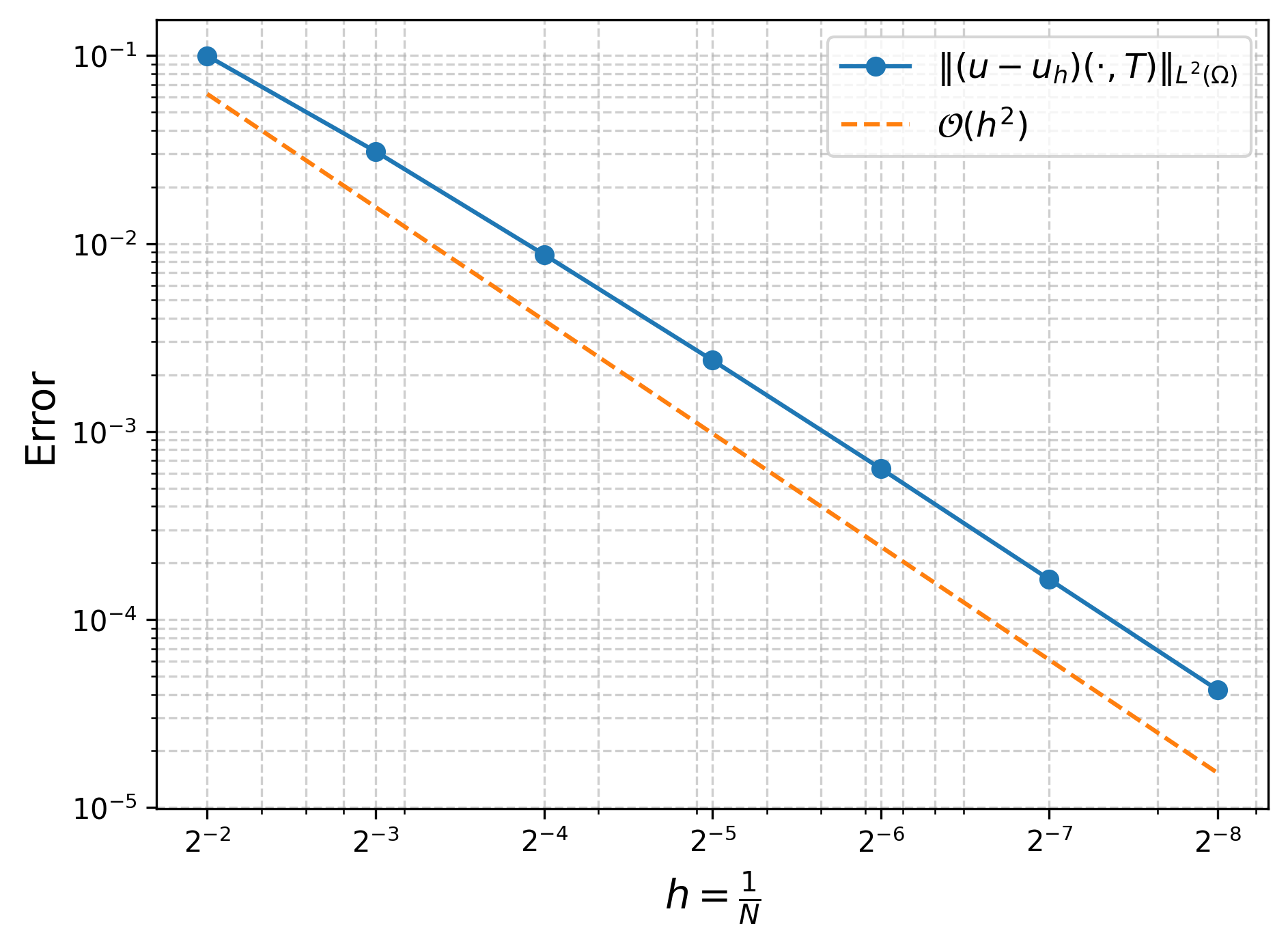}
    \end{minipage}%
    \hfill
    \begin{minipage}{0.48\textwidth}
        \centering
        \includegraphics[width=\linewidth]{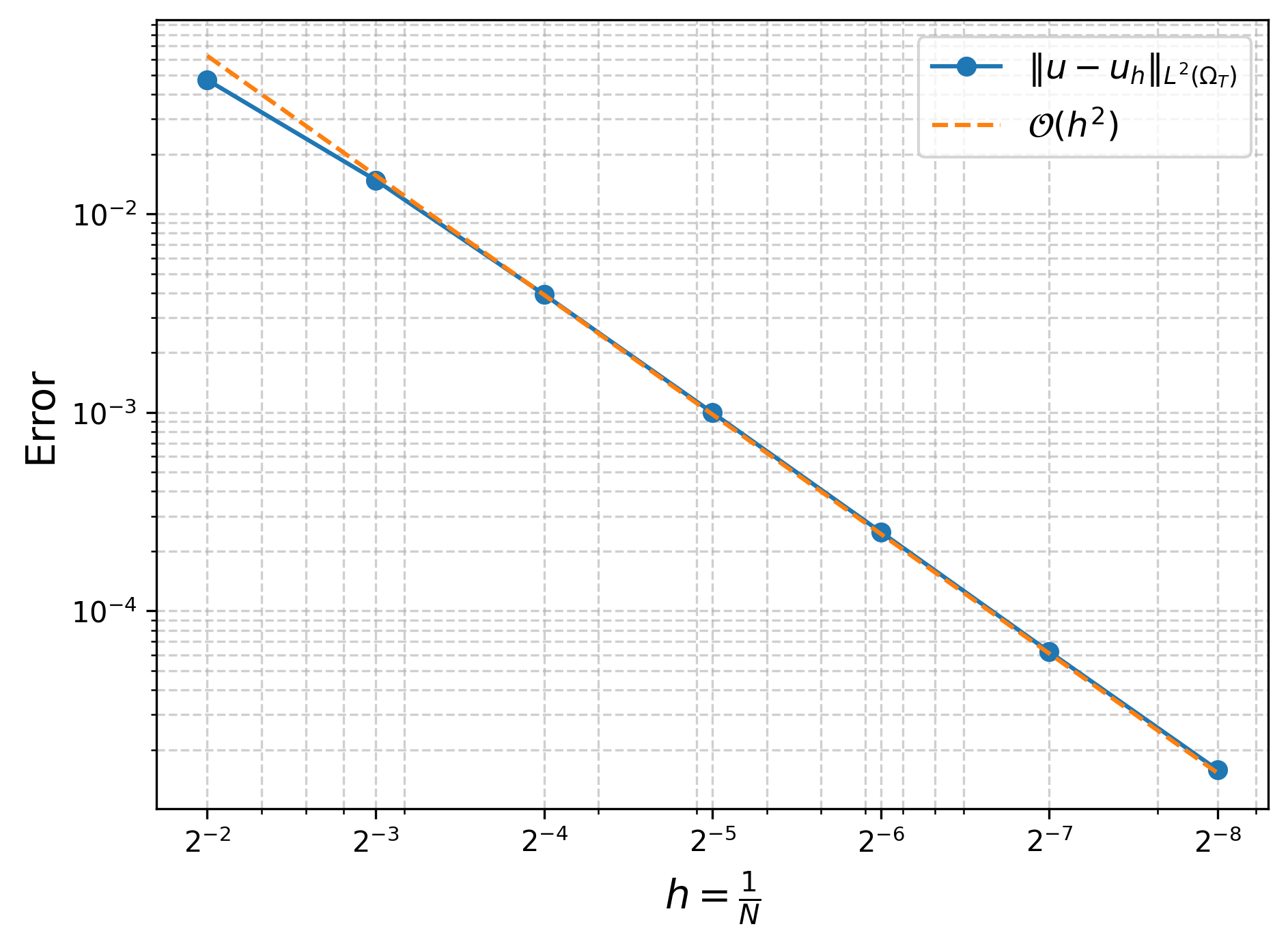}
    \end{minipage}\\[10pt]
    \centering
    \begin{minipage}{0.96\textwidth}
        \centering
        \includegraphics[width=0.5\linewidth]{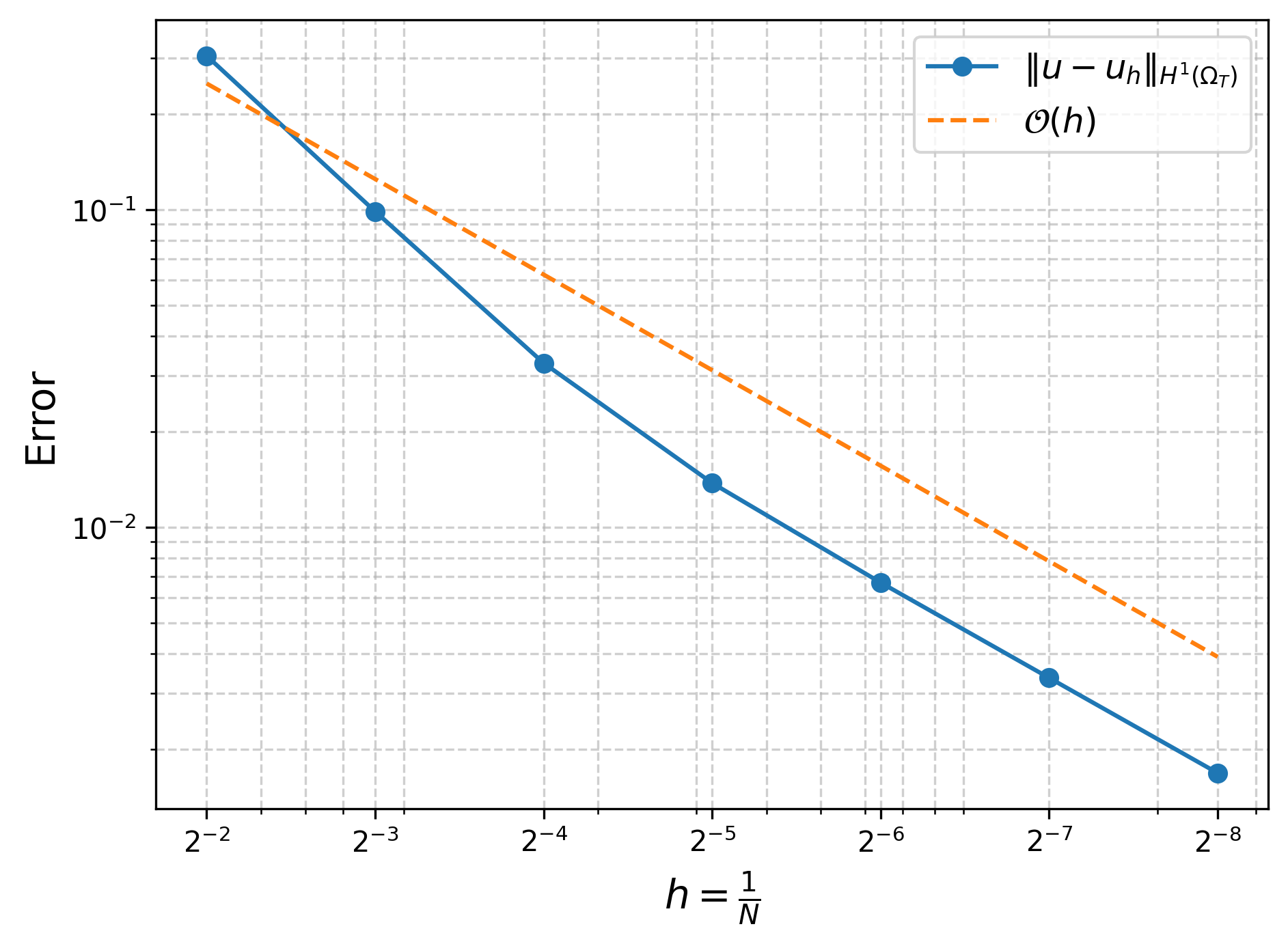}
    \end{minipage}
    \caption{Estimated convergence orders with respect to three different norms in Example \ref{ex: example 1}.}
    \label{1d example 1 convergence}
\end{figure}

\begin{example}
    \label{ex: example 2}
    In this example, we examine a homogeneous initial-boundary value problem in the space-time cylinder $\Omega_T = \left(0,1\right)^3$. The exact solution is chosen as
    $$
    u\left(\xb, t\right) = \sin\left(\pi x\right) \sin\left(\pi y\right) \sin\left(\pi t\right)\qqqq \xb = \left(x, y\right)^{\top}.
    $$
\end{example}

Table \ref{2d example 1} and Figure \ref{2d example 1 convergence} present the errors in all three norms along with the computed orders of convergence for various mesh sizes. Compared to Example \ref{ex: example 1}, increasing the number of spatial dimensions generally results in larger errors. However, the difference in errors between the two examples remains negligible. Moreover, in this example, the error estimates are still maintained at an optimal order in all three norms. These findings further support the theoretical results achieved in Section \ref{sec: a priori error estimates}.

\begin{table}[!htp]
    \centering
    \resizebox{\textwidth}{!}{
    \begin{minipage}{\textwidth}%
        \caption{Errors in three different norms for various levels of mesh refinement in Example \ref{ex: example 2}.}
        \label{2d example 1}
        \begin{tabular}{|c|c|c|c|c|c|c|}
            \hline
            \rule{0pt}{2.75ex}  
             $h$ & \( \norm{\left(u-u_h\right)\left(\cdot, T\right)}_{\Ls^2\left(\Omega\right)} \) & Order & \( \norm{u-u_h}_{\Ls^2\left(\Omega_T\right)} \) & Order & \( \norm{u-u_h}_{\Hs^1\left(\Omega_T\right)} \) & Order \\[1ex] \hline
            \rule{0pt}{2.75ex}
            \( 2^{-2} \) & \(5.905 \times 10^{-2}\) & -      & \(6.500 \times 10^{-2}\) & -      & \(4.994 \times 10^{-1}\) & -       \\ 
            \( 2^{-3} \) & \(3.607 \times 10^{-2}\) & 0.711  & \(2.182 \times 10^{-2}\) & 1.573  & \(2.094 \times 10^{-1}\) & 1.254   \\ 
            \( 2^{-4} \) & \(1.421 \times 10^{-2}\) & 1.344  & \(5.926 \times 10^{-3}\) & 1.881  & \(9.906 \times 10^{-2}\) & 1.080   \\ 
            \( 2^{-5} \) & \(4.361 \times 10^{-3}\) & 1.704  & \(1.504 \times 10^{-3}\) & 1.978  & \(5.050 \times 10^{-2}\) & 0.972   \\ 
            \( 2^{-6} \) & \(1.171 \times 10^{-3}\) & 1.897  & \(3.765 \times 10^{-4}\) & 1.998  & \(2.561 \times 10^{-2}\) & 0.980   \\ \hline
        \end{tabular}
    \end{minipage}}
\end{table}

\begin{figure}[h!]
    \centering
    \begin{minipage}{0.48\textwidth}
        \centering
        \includegraphics[width=\linewidth]{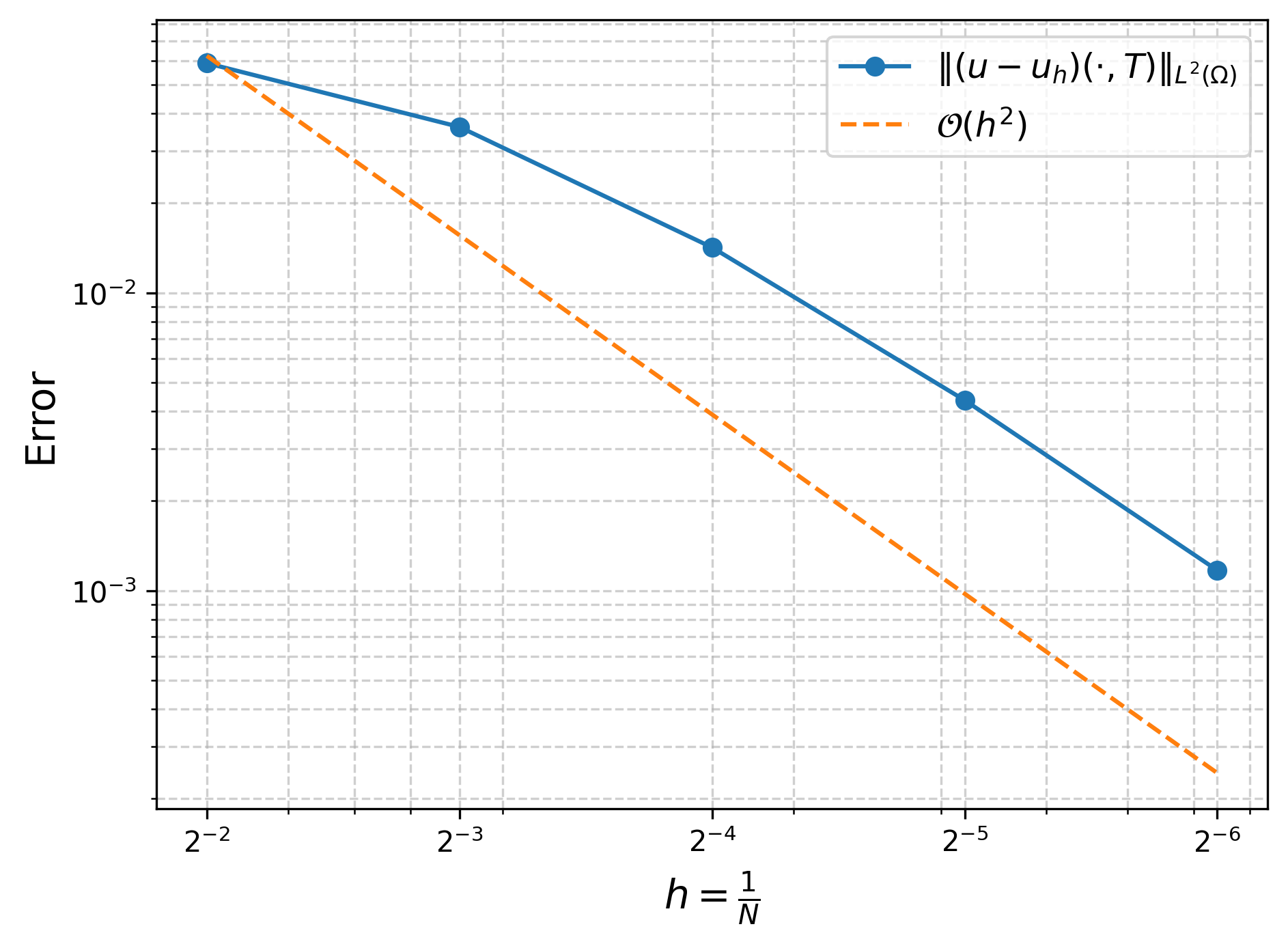}
    \end{minipage}%
    \hfill
    \begin{minipage}{0.48\textwidth}
        \centering
        \includegraphics[width=\linewidth]{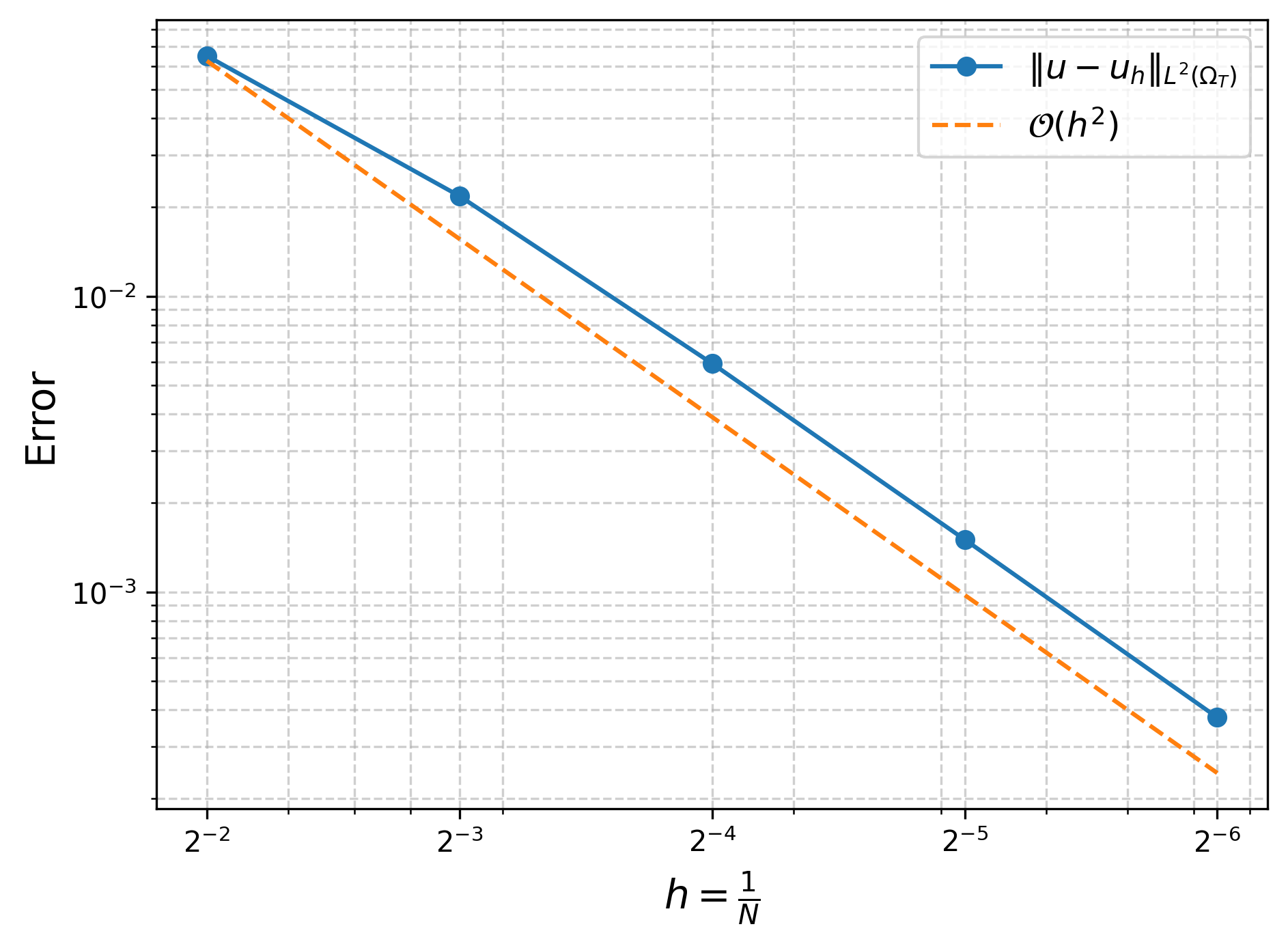}
    \end{minipage}\\[10pt]
    \centering
    \begin{minipage}{0.96\textwidth}
        \centering
        \includegraphics[width=0.5\linewidth]{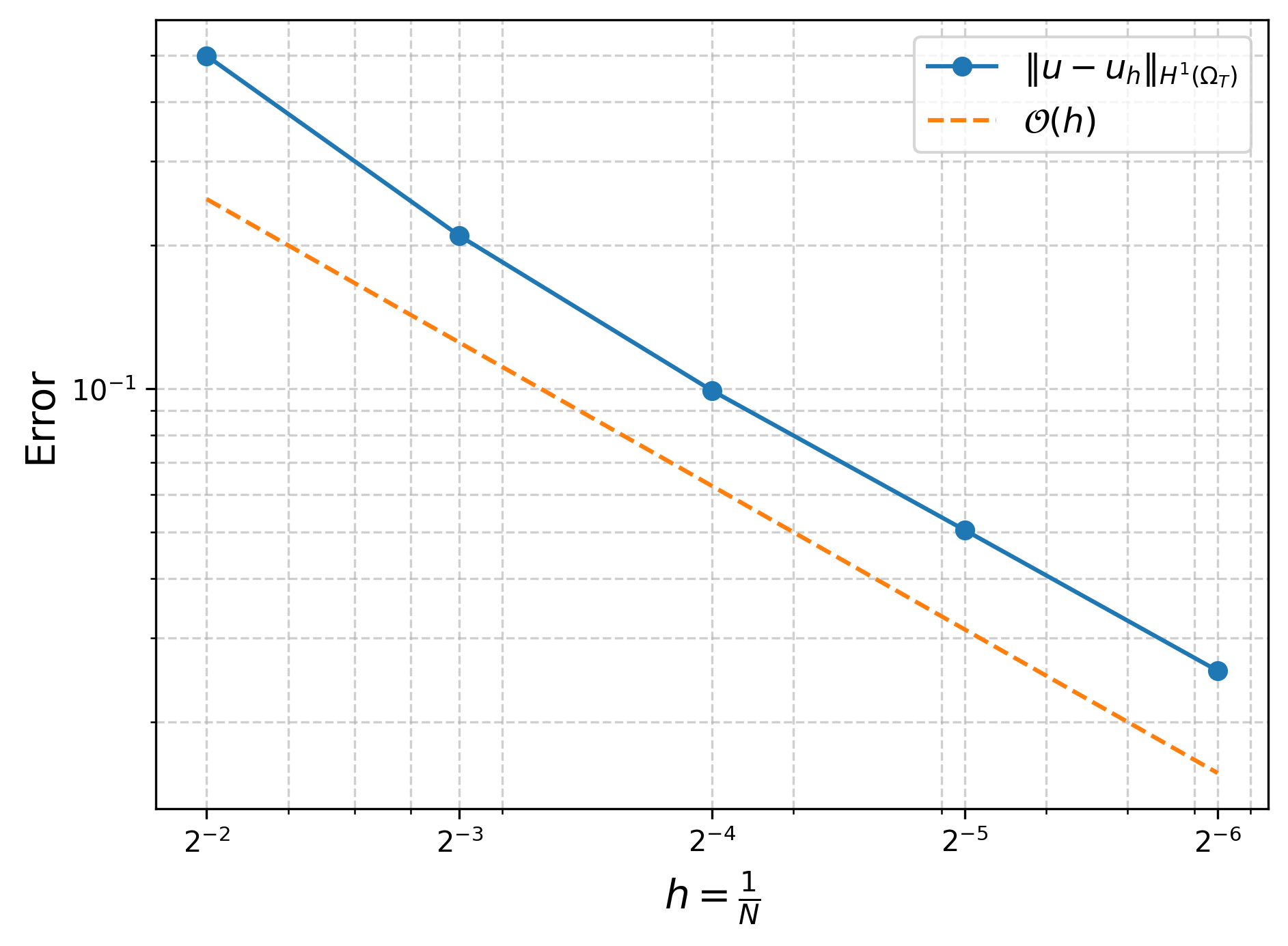}
    \end{minipage}
    \caption{Estimated convergence orders with respect to three different norms in Example \ref{ex: example 2}.}
    \label{2d example 1 convergence}
\end{figure}

\begin{example}
    \label{ex: example 3}
    This example is a modification of \cite[Example 1]{LST+2021}. We again consider $\Omega_T = \left(0,1\right)^3$ and choose the exact solution as
    $$
    u\left(\xb, t\right) = \sin\left(\pi x\right) \sin\left(\pi y\right) \left(-\dfrac{2\pi^2 + 1}{2\pi^2 + 2}t^2 +t\right) \qqqq \xb = \left(x, y\right)^{\top}.
    $$
\end{example}

Table \ref{2d example 2} and Figure \ref{2d example 2 convergence} provide the numerical results along with corresponding illustrations. These results align with those of Example \ref{ex: example 2} and further confirm the convergence behavior described in Corollary \ref{cor: order convergence in 3 norms}. The main distinction from the previous example is a slight reduction in errors at all levels of mesh refinement, particularly in $\norm{\left(u-u_h\right)\left(\cdot, T\right)}_{\Ls^2\left(\Omega\right)}$. This may be attributed to the fact that the previously chosen exact solution satisfied $u\left(\cdot, T\right) = 0$, whereas the current solution does not.

\begin{table}[!htp]
    \centering
    \resizebox{\textwidth}{!}{
    \begin{minipage}{\textwidth}%
        \caption{Errors in three different norms for various levels of mesh refinement in Example \ref{ex: example 3}.}
        \label{2d example 2}
        \begin{tabular}{|c|c|c|c|c|c|c|}
            \hline
            \rule{0pt}{2.75ex}  
             $h$ & \( \norm{\left(u-u_h\right)\left(\cdot, T\right)}_{\Ls^2\left(\Omega\right)} \) & Order & \( \norm{u-u_h}_{\Ls^2\left(\Omega_T\right)} \) & Order & \( \norm{u-u_h}_{\Hs^1\left(\Omega_T\right)} \) & Order \\[1ex] \hline
            \rule{0pt}{2.75ex}
            \( 2^{-2} \) & \(1.007 \times 10^{-2}\) & -      & \(1.697 \times 10^{-2}\) & -      & \(1.236 \times 10^{-1}\) & -       \\ 
            \( 2^{-3} \) & \(7.716 \times 10^{-3}\) & 0.384  & \(5.705 \times 10^{-3}\) & 1.573  & \(5.416 \times 10^{-2}\) & 1.190   \\ 
            \( 2^{-4} \) & \(3.296 \times 10^{-3}\) & 1.227  & \(1.570 \times 10^{-3}\) & 1.861  & \(2.674 \times 10^{-2}\) & 1.018   \\ 
            \( 2^{-5} \) & \(1.048 \times 10^{-3}\) & 1.653  & \(4.026 \times 10^{-4}\) & 1.963  & \(1.385 \times 10^{-2}\) & 0.949   \\ 
            \( 2^{-6} \) & \(2.861 \times 10^{-4}\) & 1.873  & \(1.012 \times 10^{-4}\) & 1.993  & \(7.056 \times 10^{-3}\) & 0.973   \\ \hline
        \end{tabular}
    \end{minipage}}
\end{table}

\begin{figure}[h!]
    \centering
    \begin{minipage}{0.48\textwidth}
        \centering
        \includegraphics[width=\linewidth]{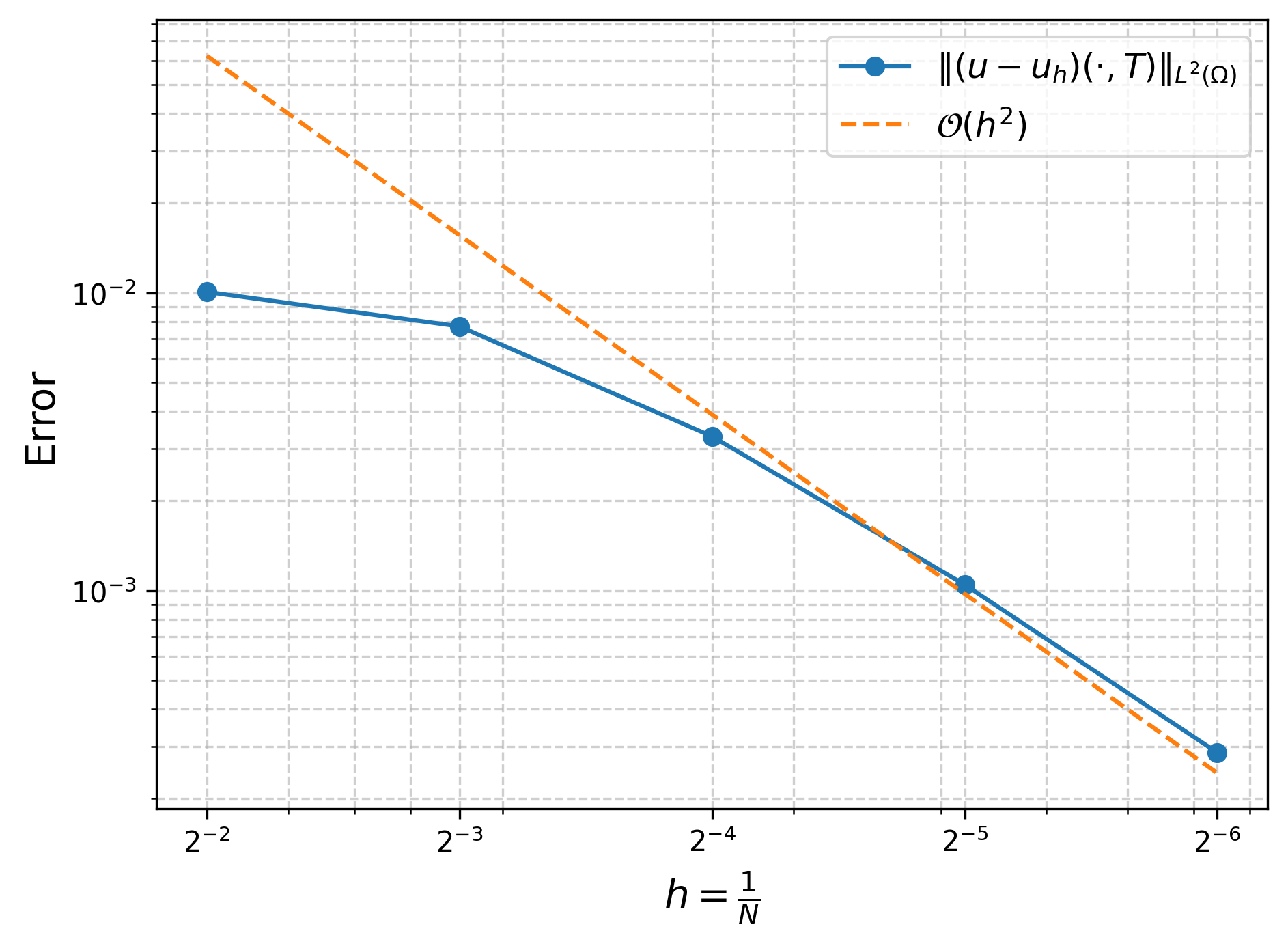}
    \end{minipage}%
    \hfill
    \begin{minipage}{0.48\textwidth}
        \centering
        \includegraphics[width=\linewidth]{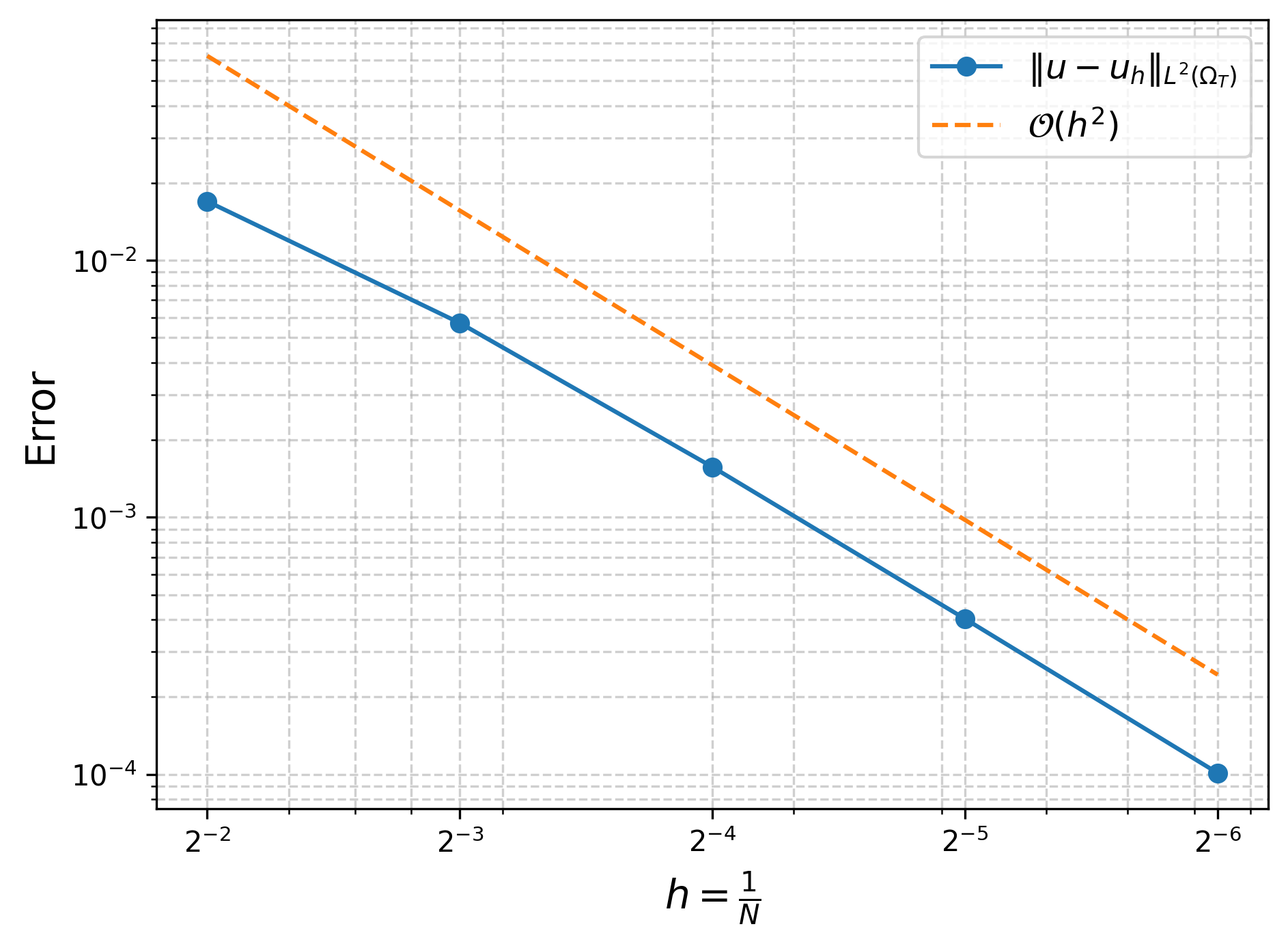}
    \end{minipage}\\[10pt]
    \centering
    \begin{minipage}{0.96\textwidth}
        \centering
        \includegraphics[width=0.5\linewidth]{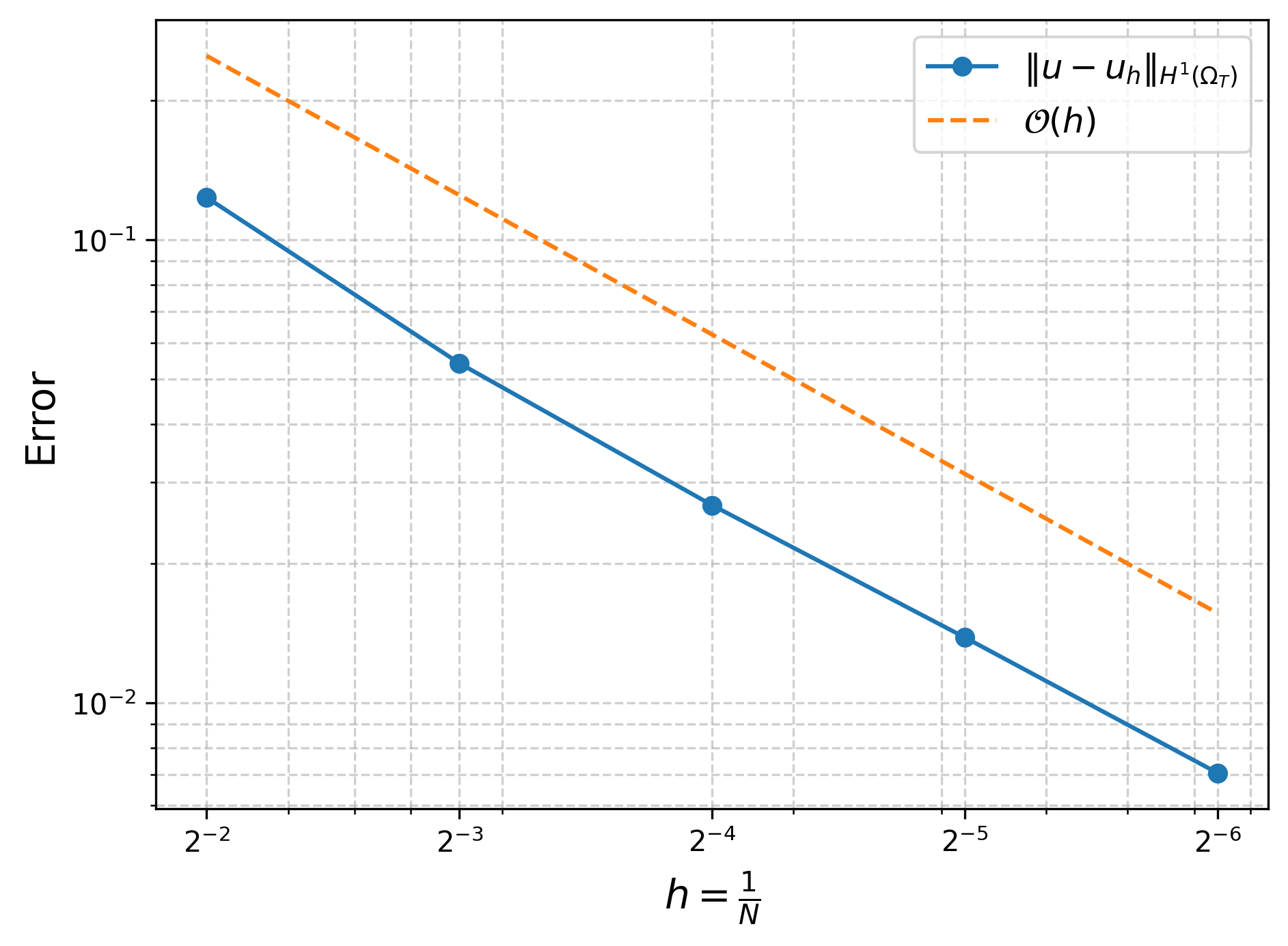}
    \end{minipage}
    \caption{Estimated convergence orders with respect to three different norms in Example \ref{ex: example 3}.}
    \label{2d example 2 convergence}
\end{figure}

\begin{example}
    \label{ex: example 4}
    Finally, we investigate a scenario in which the spatial domain evolves over time. Our main objective is to illustrate the improved error estimates for moving-domain problems. Since these theoretical results are derived using fully unstructured space-time meshes, we expect them to remain valid in this case. We consider the space-time domain $\Omega_T = \Omega\left(t\right)\times \left(0,1\right)$, where
    $$
    \Omega\left(t\right) = \left\{x\in \mathbb{R}\mid -t < x < 0.5 +t\right\}\qqqq t \in \left(0,1\right),
    $$
    and choose the exact solution as $u\left(x,t\right) = \sin\left(\pi x\right) t e^{2t}$. Evidently, this formulation leads to a problem with a non-homogeneous boundary condition.
\end{example}

Figure \ref{1d example 2 plot} presents the discrete solution $u_h$ for $h=2^{-8}$. Table \ref{1d example 2} summarizes the numerical results, showing a systematic reduction in errors with mesh refinement. Figure \ref{1d example 2 convergence} visualizes the convergence behavior, demonstrating the alignment between computed and expected convergence rates. 

These findings suggest that the established error estimates may also hold for moving-domain problems, highlighting the potential for extending the presented theoretical analysis to a broader framework.

\begin{figure}[h!]
    \centering
    \includegraphics[width=0.65\linewidth]{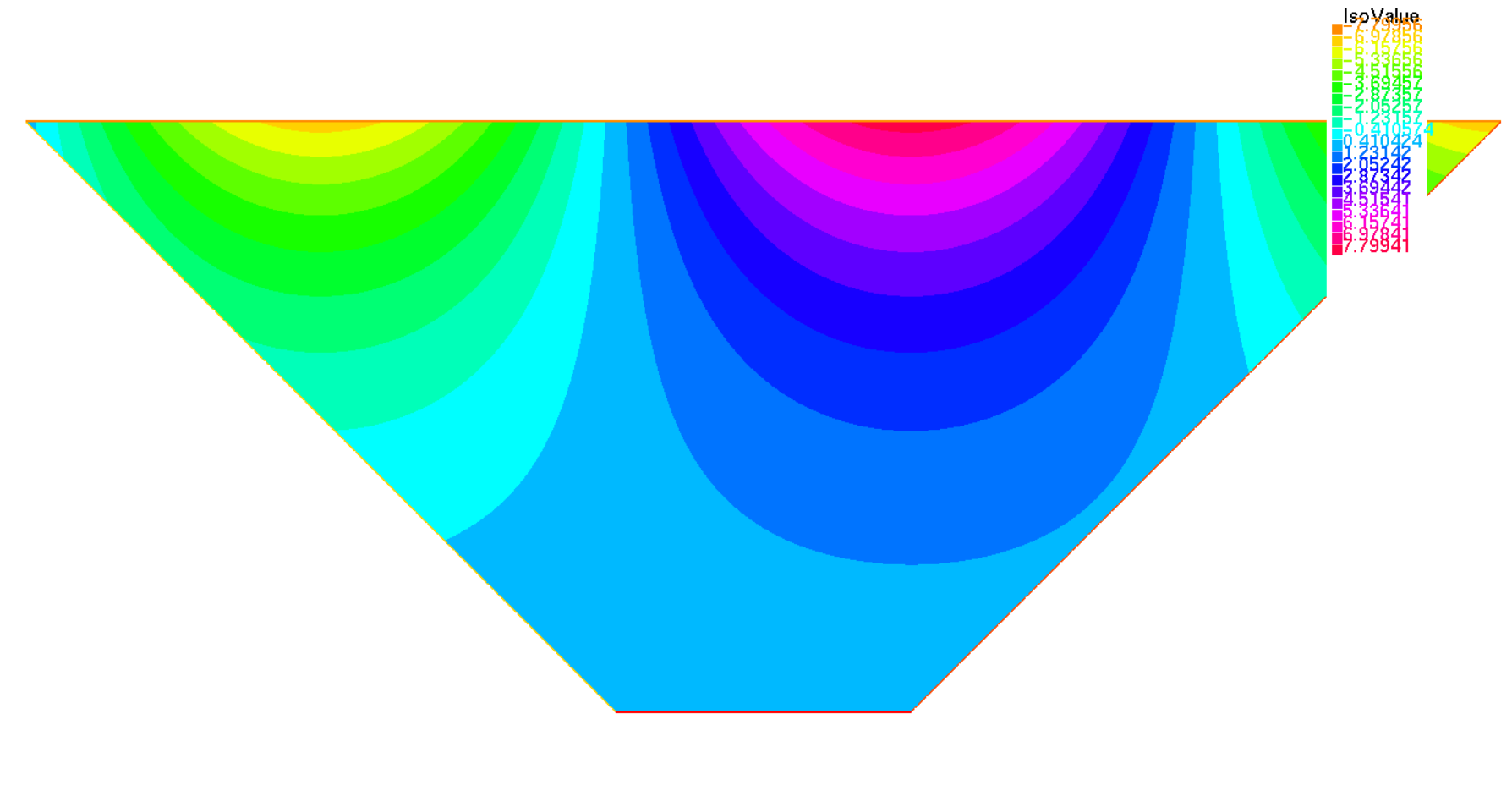}
    \caption{The discrete solution $u_h$ for Example \ref{ex: example 4} with mesh size $h= 2^{-8}$.}
    \label{1d example 2 plot}
\end{figure}

\begin{table}[!htp]
    \centering
    \resizebox{\textwidth}{!}{
    \begin{minipage}{\textwidth}%
        \caption{Errors in three different norms for various levels of mesh refinement in Example \ref{ex: example 4}.}
        \label{1d example 2}
        \begin{tabular}{|c|c|c|c|c|c|c|}
            \hline
            \rule{0pt}{2.75ex}  
             $h$ & \( \norm{\left(u-u_h\right)\left(\cdot, T\right)}_{\Ls^2\left(\Omega\right)} \) & Order & \( \norm{u-u_h}_{\Ls^2\left(\Omega_T\right)} \) & Order & \( \norm{u-u_h}_{\Hs^1\left(\Omega_T\right)} \) & Order \\[1ex] \hline
            \rule{0pt}{2.75ex}
            \(2^{-2}\) & \(7.680 \times 10^{-1}\) & -        & \(2.164 \times 10^{-1}\) & -        & \(2.560 \times 10^{0}\) & -        \\ 
            \(2^{-3}\) & \(2.207 \times 10^{-1}\) & 1.799    & \(4.630 \times 10^{-2}\) & 2.225    & \(1.225 \times 10^{0}\) & 1.064    \\ 
            \(2^{-4}\) & \(6.813 \times 10^{-2}\) & 1.696    & \(1.277 \times 10^{-2}\) & 1.859    & \(5.580 \times 10^{-1}\) & 1.134    \\ 
            \(2^{-5}\) & \(1.787 \times 10^{-2}\) & 1.930    & \(3.431 \times 10^{-3}\) & 1.896    & \(2.827 \times 10^{-1}\) & 0.981    \\ 
            \(2^{-6}\) & \(5.008 \times 10^{-3}\) & 1.836    & \(1.016 \times 10^{-3}\) & 1.756    & \(1.405 \times 10^{-1}\) & 1.009    \\ 
            \(2^{-7}\) & \(1.283 \times 10^{-3}\) & 1.965    & \(2.966 \times 10^{-4}\) & 1.776    & \(7.014 \times 10^{-2}\) & 1.002    \\ 
            \(2^{-8}\) & \(3.440 \times 10^{-4}\) & 1.899    & \(8.480 \times 10^{-5}\) & 1.807    & \(3.332 \times 10^{-2}\) & 1.074    \\ \hline
        \end{tabular}
    \end{minipage}}
\end{table}

\begin{figure}[h!]
    \centering
    \begin{minipage}{0.48\textwidth}
        \centering
        \includegraphics[width=\linewidth]{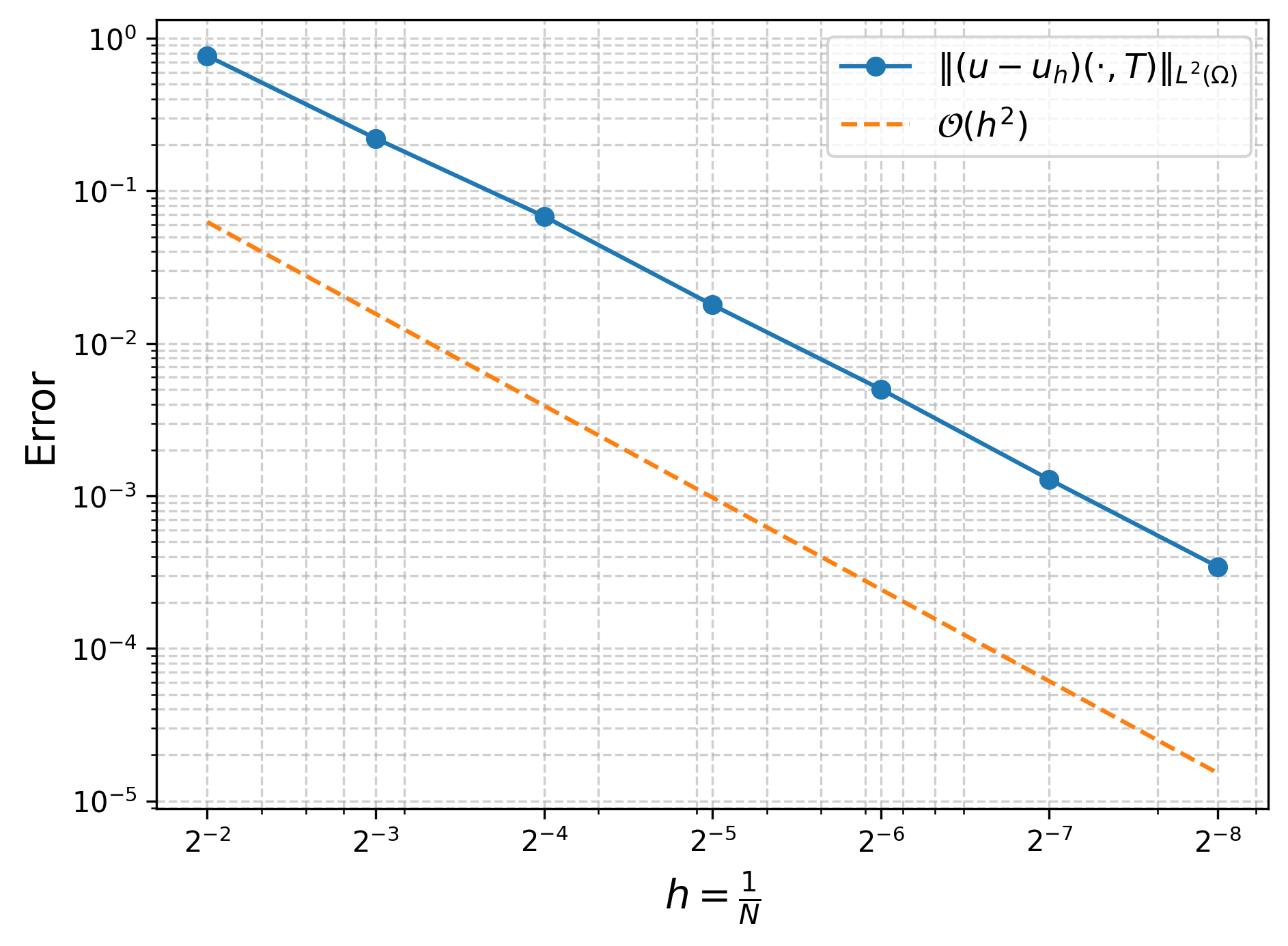}
    \end{minipage}%
    \hfill
    \begin{minipage}{0.48\textwidth}
        \centering
        \includegraphics[width=\linewidth]{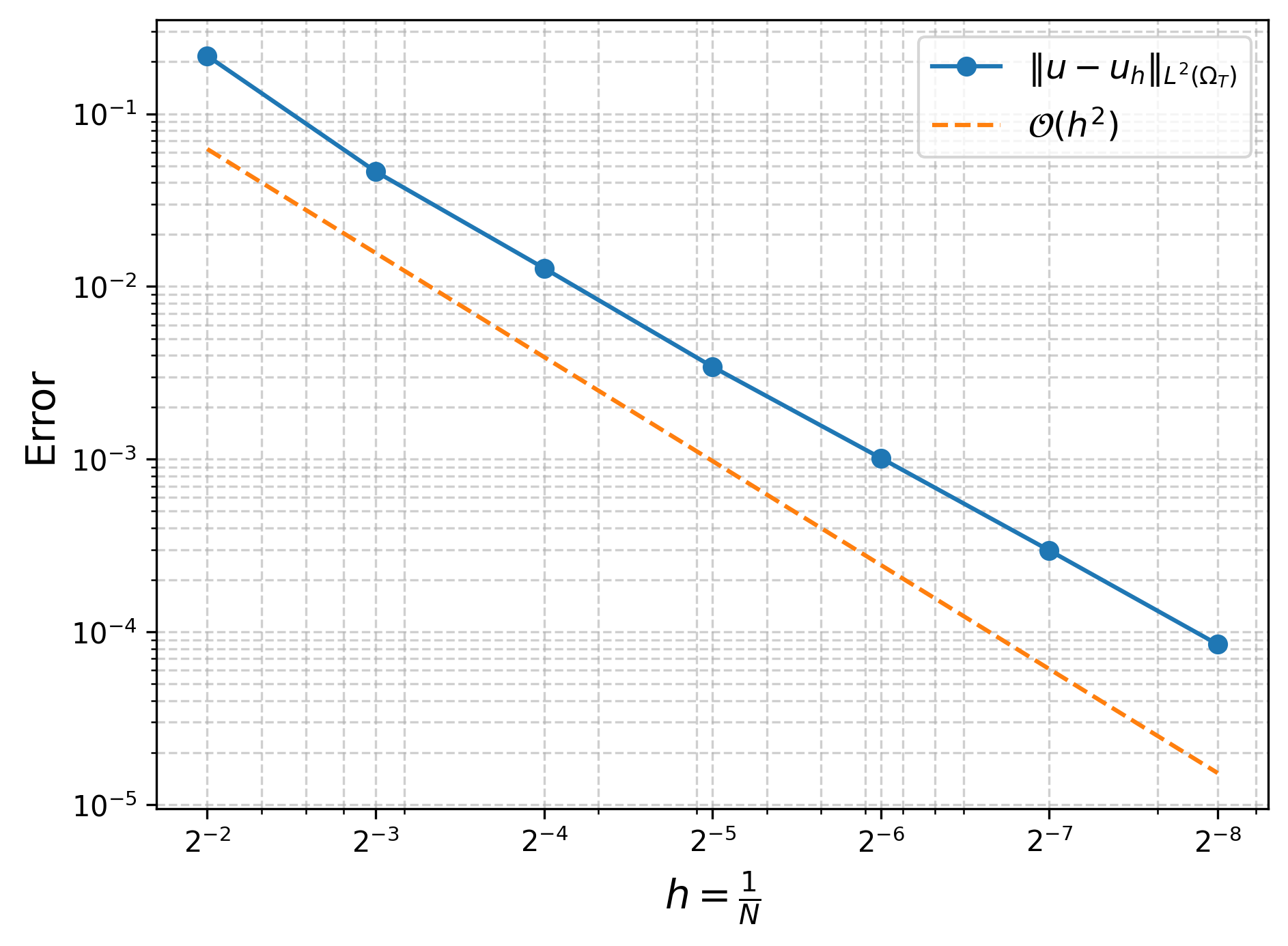}
    \end{minipage}\\[10pt]
    \begin{minipage}{0.48\textwidth}
        \centering
        \includegraphics[width=\linewidth]{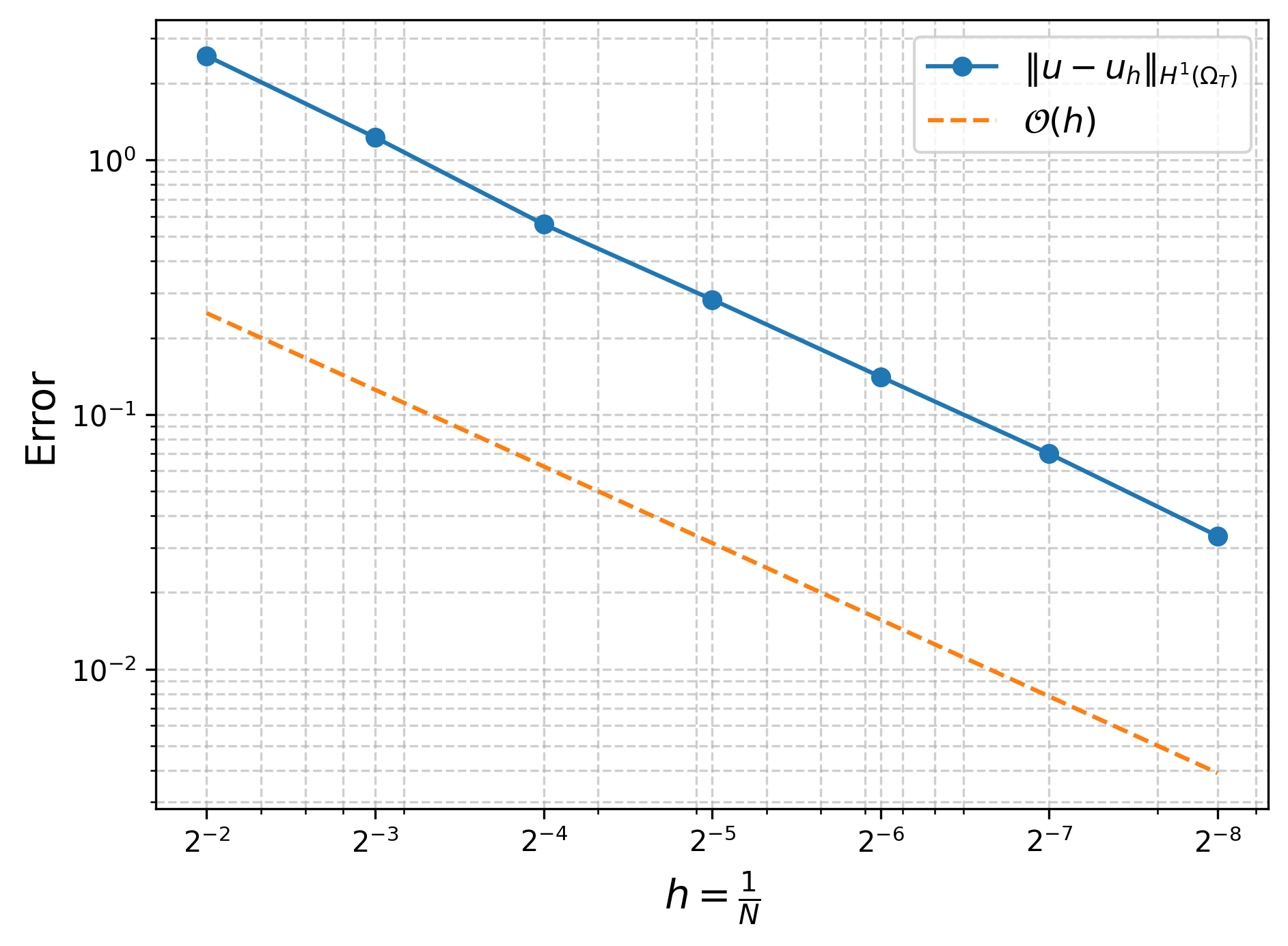}
    \end{minipage}
    \caption{Estimated convergence orders with respect to three different norms in Example \ref{ex: example 4}.}
    \label{1d example 2 convergence}
\end{figure}

\section{Conclusions}

Using duality arguments, we derived new error estimates for the space-time finite element method. Compared to the trial space norm, employing three weaker norms led to higher-order estimates under appropriate conditions. Conversely, for a stronger norm, the convergence order remained consistent with the mesh-dependent norm in \cite{Steinbach2015}. Finally, we provided numerical examples in one and two spatial dimensions to illustrate these findings.

For future work, one possible direction is to investigate superconvergence based on the established negative-order norm error estimate. Another potential direction is to extend the proposed approach to enhance error estimates for other space-time methods, such as isogeometric analysis and the mixed finite element method.

\backmatter




\bibliography{main}

\end{document}